\documentclass[12pt,reqno]{amsart}

\usepackage{amsmath,amsfonts,amsthm,graphicx,float,color}
\usepackage{cite}
\usepackage[hidelinks]{hyperref}
\usepackage{xcolor}
\hypersetup{
    colorlinks,
    linkcolor={blue!80!black},
    citecolor={blue!80!black},
    urlcolor={blue!80!black}
}

\usepackage{cancel}
\usepackage[margin = 1in]{geometry}
\usepackage{mathrsfs}
\usepackage{mathabx}
\usepackage{dsfont}
\usepackage[normalem]{ulem}

\newcommand{\N}{\mathbb{N}}

\newcommand{\R}{\mathbb{R}}
\newcommand{\Z}{\mathbb{Z}}

\newcommand{\ve}{\varepsilon}

\newcommand{\lp}{\left(}
\newcommand{\rp}{\right)}

\newcommand{\wt}{\widetilde}
\newcommand{\wh}{\widehat}

\newcommand{\Op}{\textnormal{Op}}

\newcommand{\supp}{\textnormal{supp\,}}

\newcommand{\floor}[1]{\lfloor#1\rfloor}

\newcommand{\cbar}{\overline{c}}
\newcommand{\tbar}{\overline{t}}
\newcommand{\mubar}{{\overline{\mu}}}
\newcommand{\mbar}{{\overline{m}}}
\newcommand{\lambdabar}{{\overline{\lambda}}}
\newcommand{\Pbar}{{\overline{P}}}
\DeclareMathOperator{\Hess}{Hess}

\def\ox{\overline{x}}
\def\oxi{\overline{\xi}}

\makeatletter
\def\thmhead@plain#1#2#3{%
  \thmname{#1}\thmnumber{\@ifnotempty{#1}{ }\@upn{#2}}%
  \thmnote{ {\the\thm@notefont#3}}}
\let\thmhead\thmhead@plain
\makeatother

\newtheorem{thm}{Theorem}
\newtheorem{lem}{Lemma}
\newtheorem{prop}[lem]{Proposition}
\newtheorem{cor}[lem]{Corollary}
\newtheorem{rmk}[lem]{Remark}

\newtheorem{defn}[lem]{Definition}

\newcommand{\thmref}[1]{Theorem~\ref{#1}}

\newcommand{\lemref}[1]{Lemma~\ref{#1}}

\definecolor{bpurple}{RGB}{170,0,200}
\definecolor{bgreen}{RGB}{0,170,120}

\numberwithin{equation}{section}
\numberwithin{lem}{section}

\def\R{\mathbb R}

\def\beas{\begin{eqnarray*}}
\def\eeas{\end{eqnarray*}}

\newcommand{\spann}{\textnormal{span\,}}
\def\tO{\widetilde{\Omega}}

\begin{document}

\title[Pointwise Weyl Laws for QCI Systems]{Pointwise Weyl Laws for Quantum Completely Integrable Systems}

\author[S. Eswarathasan]{Suresh Eswarathasan}
\address[S.~ Eswarathasan]{Department of Mathematics and Statistics, Dalhousie University, Halifax, NS, Canada B3H 4R2}
\email{sr766936@dal.ca}

\author[A. Greenleaf]{Allan Greenleaf}
\address[A.~ Greenleaf]{Department of Mathematics, University of Rochester, Rochester, NY, USA 14627}
\email{allan@math.rochester.edu}

\author[B. Keeler]{Blake Keeler}
\address[B.~ Keeler]{Department of Mathematical Sciences, Montana Technological University, Butte, MT, USA 59701}
\email{bkeeler@mtech.edu}

\maketitle

\begin{abstract}
The study of the asymptotics  of the spectral function for self-adjoint, elliptic differential,
or more generally pseudodifferential, operators on a compact 
manifold has a long history. The seminal 1968 paper of H\"ormander, following important prior 
contributions by G\"arding, Levitan, Avakumovi\'c, and Agmon-Kannai (to name only some), obtained pointwise asymptotics (or 
 a ``pointwise Weyl law") for a single elliptic, self-adjoint operator. 
 
Here, we establish a microlocalized pointwise Weyl law for the 
joint spectral functions of {\it quantum completely integrable}  (QCI) systems,
$\overline{P}=(P_1,P_2,\dots, P_n)$, where $P_i$ are first-order, classical, self-adjoint, pseudodifferential operators
on a compact manifold $M^n$,
with $\sum P_i^2$ elliptic and $[P_i,P_j]=0$ for $1\le i,j\le n$.
A particularly important case is when $(M,g)$ is Riemannian and $P_1=(-\Delta)^\frac12$.
We illustrate our result with  several examples, including surfaces of revolution.

\end{abstract}


\tableofcontents

\section{Introduction}
\subsection{Pointwise/Off-Diagonal Weyl Laws}
Let $(M,g)$ be a compact, connected, boundaryless $n$-dimensional Riemannian manifold, with (positive) Laplace-Beltrami operator acting
on functions,  $-\Delta_g$. It is well-known that $-\Delta_g$ has
a discrete spectrum, consisting of eigenvalues $0 = \lambda_0^2 <\lambda_1^2 \le \lambda_2^2 \le \dots$,
and  an orthonormal basis $\{\phi_j\}_{j=0}^\infty$ of $L^2(M)$ consisting of real-valued eigenfunctions, each satisfying 
\[-\Delta_g \phi_j = \lambda_j^2\phi_j,\, \forall j\ge 0.\]
The nature of the Laplace spectrum has been the subject of a vast amount of work over the last hundred 
years, particularly in the  high-frequency regime, where one studies the spectral behavior as $\lambda_j\to\infty.$ 
Perhaps the  most fundamental such result is Weyl's law (often referred to as just ``the Weyl law"), which gives an asymptotic for the growth of the eigenvalue counting function, 
$N(\lambda) := \#\{j:\,\lambda_j\le\lambda\}$,  as $\lambda\to\infty$. In particular, there exists $C_g>0$ so that one has 
\begin{equation}\label{e:weyl_counting}
N(\lambda) = C_g\lambda^n +\mathcal O(\lambda^{n-1}),
\end{equation}
and this result is sharp without additional assumptions on $(M,g)$, as can be seen from the example of the round sphere, 
$\mathbb S^n.$

More generally, one can study, for $\lambda>0$, the \emph{spectral function}, 
defined as the Schwartz kernel of the orthogonal projection 
\[\Pi_\lambda:\bigoplus\limits_{\lambda_j\le\lambda}\ker(\Delta_g + \lambda_j^2),\]
which in the above basis takes the form
\[\Pi_{\lambda}(x,y) =\sum\limits_{\lambda_j\le\lambda}\varphi_j(x)\varphi_j(y).\]
This generalizes the eigenvalue counting function $N(\lambda)$  in the sense that if we compute the trace
and use the fact that the eigenfunctions are orthonormal, we obtain
\[\textnormal{trace}(\Pi_\lambda) = \int\limits_{M}\Pi_\lambda(x,x)\,dv_g(x) = \sum\limits_{\lambda_j\le\lambda} 1 = N(\lambda).\]

Now, let $inj(M)$ denote the injectivity radius of $M$.  The spectral function also has a well-known asymptotic as $\lambda\to\infty$, often 
referred to as a ``pointwise" Weyl law, given by 
\begin{equation}\label{e:weyl_off-diagonal}
\Pi_\lambda(x,y) = \frac{\lambda^n}{(2\pi)^n}\int\limits_{S_y^*M}e^{i \lambda \langle\exp_y^{-1}(x),\xi\rangle_{g_y}}\,\frac{d\xi}{\sqrt{|g_y|}} 
+ \mathscr{R}(x,y,\lambda),
\end{equation}
where there exists an $\ve \in (0, inj(M))$ such that 
\[\sup\limits_{d(x,y)\le \ve}\left|\mathscr{R}(x,y,\lambda)\right| =\mathcal O(\lambda^{n-1})\]
as $\lambda\to \infty$. This result is due to H\"ormander \cite{Hormander1968},  building on earlier results of 
Avakumovi\'c and Levitan \cite{Avakumovic1956,Levitan1952,Levitan1955}. H\"ormander's result is in fact more general than this, 
holding for any elliptic pseudodifferential operator. As with the eigenvalue counting function, the remainder bound here is sharp, which 
can again be seen on $\mathbb S^n$. If one is willing to impose restrictions on the geometry of $M$, it is possible to obtain 
improvements in both \eqref{e:weyl_counting} and \eqref{e:weyl_off-diagonal}. A great deal of research has been done in this 
area; see, e.g.,  \cite{Safarov1988,SoggeZelditch2002,CanzaniHanin2015,CanzaniHanin2018,Berard1977,Keeler2023}, which is 
by no means an exhaustive list.

Our main result, Theorem \ref{thm:main}, to be described in Section \ref{sect:setup_main}, can be thought as a generalization of H\"ormander's (\ref{e:weyl_off-diagonal}) for multiple commuting operators;   we set up the hypotheses for this in  Section \ref{sec: QCI}.

\subsection{QCI Systems and Laplacians}\label{sec: QCI}
In this article, we are interested in the behavior of the  joint spectral function of a  \emph{quantum completely integrable (QCI)}
system of pseudodifferential operators on an $n$-dimensional manifold $M$ without boundary.
{All pseudodifferential operators are classical (or {\it polyhomogeneous}); if $P(x,D)\in\Psi^m(M)$, 
its principal symbol $p(x,\xi)=\sigma(P)(x,\xi)$ is homogeneous of degree $m$ on $T^*M\setminus 0$.}
We assume that $M$ is equipped with a volume form, $dx:=d \mbox{Vol}$;
this includes the case when $(M,g)$ is Riemannian, with $dx=d \mbox{Vol}_g$.

\begin{defn} \label{defn:QCI}
(a) A \emph{quantum completely integrable (QCI) system} on $M$
is a  collection $\overline{P}$  of $n$  self-adjoint, first order pseudodifferential operators,
 $\overline{P}=(P_1, \dots, P_n)$, 
satisfying\smallskip

(i) $\sum_{i=1}^n P_i^2$ is elliptic; 
\smallskip

(ii)  the \emph{Liouville integrability condition}, consisting of 
\begin{equation}\label{eqn:QCIi}  
[P_i, P_j] = 0, \, \, \, \forall i,j=1,\dots, n,
\end{equation}
 and  the nondegeneracy of their (real) principal symbols $p_i=\sigma(P_i)$,
\begin{equation}\label{eqn:QCIii}
d_{x,\xi}p_1 \wedge \dots \wedge d_{x,\xi}p_n \neq 0,
\end{equation} 
on a dense, open, conic subset $\Omega_0 \subset T^*M \backslash 0$.
\smallskip  

(b)     A Riemannian manifold $(M,g)$ is said to be \emph{QCI  }   (or its Laplace-Beltrami operator $\Delta_g$ is said to be QCI) 
    if one can extend $P_1:=\sqrt{-\Delta_g}$ to a QCI system $\overline{P}$ on $M$.
    \end{defn}

This paper   concerns the joint spectral asymptotics of QCI systems, microlocalized to  
regions of phase space exhibiting a stronger non-degenerate behavior than \eqref{eqn:QCIii}.
Recall that for a function $p:T^*M\to\R$, although the full gradient $d_{x,\xi}p$ is  intrinsically defined,
 $d_xp$ and  $d_\xi p$  are not without further structure being imposed, e.g.,  fixing a local coordinate system.
Let $proj_M:T^*M\to M$ denote the natural projection, which is an open map.

\begin{defn} \label{def:fiber_rank}
Suppose  $1\le k\le n$.
We say that a set  $\{p_1, \dots, p_k\}$ of functions, real-valued and homogeneous of degree 1,
has \emph{fiber rank} $k$ on an open conic subset $\tO \subset T^*M\setminus 0$ 
if, for any $(x_0,\xi_0)\in\tO$, there exists a conic neighborhood $U_{x_0,\xi_0}$ of $(x_0,\xi_0)$ 
and  a coordinate system $\ox$ on the  open set $proj_M(U_{x_0,\xi_0})\owns x_0$
such that, with respect to the induced canonical coordinates $(\ox,\oxi)\in T^*\R^n\setminus 0$ on $U_{x_0,\xi_0}$,
$\dim \spann  \{ \nabla_{\oxi} p_1, \dots, \nabla_{\oxi} p_n \} =k$. 
\end{defn}

\begin{rmk}\label{rmk:def1.2}
\emph{On $M=\R^n$, the fiber rank $k$ condition is part of an admissibility condition   
introduced by Tacy \cite[Thm. 0.1 (1,2)]{Tacy2019} in
the study of $L^p$ estimates for semi-classical quasimodes. Geometrically it corresponds
to the level surfaces $\left\{\, \{\xi\, |\, p_i(x_0,\xi)=p_i(x_0,\xi_0)\}\, \right\}_{i=1}^k$ through any point $(x_0,\xi_0)\in\tO$   
being smooth in   $T^*_{x_0}M\setminus 0$ and intersecting  transversely.}

\emph{
For $k=n$, the fiber rank $n$ condition can be expressed in terms  the moment map $\overline{p}$ (defined below) in the following way.  
Let $\pi: T^*M \rightarrow M$ be the canonical projection.
Then the fiber rank $n$ condition  is equivalent with,
for  any level set  $\Sigma$  of the moment map,
$D \pi:T\Sigma \rightarrow TM$ having full rank, i.e., $\pi: \Sigma \rightarrow M$ 
is a submersion.}

\emph{Note also that for $k=n$ the condition described in Def. \ref{def:fiber_rank} is stronger than that in \eqref{eqn:QCIii} and thus,
in the context of a QCI system, such a set $\tO$ is  contained in the set $\Omega_0$ of Def. \ref{defn:QCI}(a).
We refer to a QCI system $\overline{P}$ whose principal symbols   $\{p_1,\dots,p_n\}$
satisfy the fiber rank condition with $k=n$  on $\tO$ as simply satisfying the  \emph{fiber rank condition},
and this will play a prominent role in the statement of the results in Section \ref{sect:setup_main}.}
\end{rmk}

We now  review some of the spectral theory for QCI systems   most relevant to this article.
{The spectral theorem for commuting, unbounded, self-adjoint operators, 
along with the ellipticity of $\sum_{i=1}^n P_i^2$, 
implies there is a  discrete \textit{joint spectrum} of $\overline{P}$,}
\begin{equation} \label{eqn:joint_spectrum}
\Lambda := \{ \overline{\lambda} = (\lambda^{(1)}, \dots, \lambda^{(n)})\in\R^n : 
\exists\, \phi_{\overline{\lambda}}\hbox{ s.t. } P_i \varphi_{\overline{\lambda}} =  
\lambda^{(i)}\varphi_{\overline{\lambda}},\, \forall\, 1\le i\le n \}.
\end{equation}
{Each \textit{joint eigenfunction} $\varphi_{\overline{\lambda}}$ is an eigenfunction of the 
elliptic $\sum P_i^2$ and thus is $C^\infty$.
This also shows that $\Lambda$ is countably infinite and one may write a corresponding orthonormal basis 
of joint eigenfunctions as 
$\{\varphi_{\overline{\lambda}_j} \}_{j \in \mathbb{N}}$; 
we write each $\overline{\lambda}_j \in \Lambda$ in terms of coordinates as 
$\overline{\lambda}_j=(\lambda_j^{(1)}, \dots, \lambda_j^{(n)})$.  }
\bigskip

{\it The goal of this paper is to study the pointwise asymptotics of the Schwartz kernels of the $L^2$ projectors 
associated to certain subsets of $\Lambda$.  }
 \bigskip

We recall an object that further describes $\Lambda$.  The \textit{moment  map} of the QCI system 
(or \textit{momentum map}; see \cite{Mathoverflow_Momentum})
is 
$\overline{p}: T^*M \setminus 0 \rightarrow \mathbb{R}^n \setminus 0$, defined by 
$\overline{p}(x,\xi) = \left(p_1(x,\xi), \dots, p_n(x,\xi)\right) $
for $(x,\xi) \in T^*M\setminus 0$. 
It is well-known that the moment map has a privileged position in the theory of classical integrable dynamical systems, and so it is not 
surprising that its properties continue to play an important role in the spectral theory of QCI systems. 
An example of this is its influence on 
$L^p$ norms and eigenfunction concentration; see the survey article of Toth and Zelditch \cite{TZ01} and references therein.  
Denote by
\begin{equation} \label{e:image_mm}
    \Gamma : = \overline{p}(T^*M\setminus 0)\subset \R^n \setminus 0.
\end{equation}
the image of the moment map, which is conic  due to the homogeneity of the principal symbols $p_i$.
$\Gamma$ gives a first description of the ``geometry" of $\Lambda$, a heuristic  rigorized by some fundamental results 
for QCI systems
due to Colin de Verdi\`ere, which we now describe.
\smallskip 

Let $\overline{P}$ be a QCI system, $\Gamma$ the image of its moment map as in \eqref{e:image_mm}, and suppose that
$\mathcal{C} \subset \mathbb{R}^n\setminus 0$ is a (possibly empty) conic set such that 
$\mathcal{C} \cap \Gamma = \emptyset$.  
Colin de Verdi\`ere's first fundamental result is that $\mathcal{C} \cap \Lambda$ is finite \cite[Theorem 0.6]{CdV79}.  
Since $\Gamma^{\complement}$ is also conic, 
it follows that that $\Lambda \backslash\, B\left(0,R_0\right)\subset \Gamma$ for some $R_0 >0$.

Next, let $\Gamma_{{crit}}$ denote the critical values of the moment map.  
Colin de Verdi\`ere's second fundamental result \cite[Theorem 0.7]{CdV79} 
states that for any non-empty conic set $\mathcal{C}$ with a piecewise $C^1$ boundary $\partial \mathcal{C}$ 
such that $\partial \mathcal{C} \cap \Gamma_{{crit}} = \emptyset$, 
as $\lambda\to\infty$ we have
    \begin{equation}\label{eqn:integ_Weyl}
        \# \{ j \, : \, \overline{\lambda}_j \in \mathcal{C} \cap \Lambda \, \mbox{ and } \, \|\overline{\lambda}_j \| \leq \lambda \} 
        = (2 \pi)^{-n} \mbox{Vol}_{T^*M} \left( \, \overline{p}^{-1} \left\{ \mathcal{C} \cap B(0, \lambda) \right\} \, \right)
        + \mathcal{O}(\lambda^{n-1}), 
    \end{equation}
where $\mbox{Vol}_{T^*M}$ is the  symplectic (Liouville) volume on $T^*M$.  

\subsection{Setup and statement of main results} \label{sect:setup_main}

We are now in a position to describe the main analytical object studied in this paper.
For $\cbar\in \mathbb S^{n-1}$\,  with coordinates 
\footnote{In the semiclassical analysis literature, $\overline{c}$ would usually be referred to as 
$\frac{\overline{E}}{|\overline{E}|}$,  a normalized energy vector.}
$c_i\ne 0,\, 1\le i\le n$,
and $\lambda \in \mathbb{R}_+$, define $I(\lambda, \overline{c}) \subset\R^n$ by 
$$I(\lambda, \overline{c}) := [-|c_1|\lambda, |c_1| \lambda] \times 
[-|c_2| \lambda, |c_2| \lambda] \times \dots \times [-|c_n| \lambda, |c_n| \lambda].$$
  Then the \textit{ joint spectral function of $\overline{P}$ over $I(\lambda, \overline{c})$}, is
  \begin{equation} \label{eqn:joint_spectral_function}
    \Pi_{\lambda, \cbar }(x,y):= \sum\limits_{j=0}^\infty
    \mathds 1_{I(\lambda, \overline{c})}\left(\overline{\lambda}_j \right)
     \varphi_{\overline{\lambda}_j}(x)\overline{\varphi_{\overline{\lambda}_j}(y)}\in C^\infty(M\times M).
\end{equation}
Note  that this orthonormal projector 
 has a non-trivial range for $\lambda$ large enough.

Recall the asymptotic (\ref{e:weyl_off-diagonal}).  Our main theorem is a microlocalized-QCI analogue of the main result in 
\cite{Hormander1968} and is a pointwise version of Colin de Verdi\`ere's ``integrated" Weyl Law for QCI systems.  For a list of 
notation used, see Section \ref{sect:list}.

\begin{thm}\label{thm:main}
Suppose that $\overline{P}=\left(P_1,P_2,\dotsc,P_n\right) $ is a QCI system 
on   a smooth, compact  manifold without boundary, $M$,
which satisfies the fiber rank $n$ condition on  a conic, open subset
$\tO \subset T^*M\setminus 0$. 
Let $\cbar\in \overline{p}(\tO)$, with $c_i\ne 0,\, \forall\, 1\le i\le n$.
For any $(x_0,\xi_0)\in\tO$, let $U_{x_0,\xi_0}\subset\tO$ and $(x,\xi)$ be local canonical coordinates, 
referred to as $(\ox,\oxi)$ in Def. \ref{def:fiber_rank}.
Finally, assume that $M$ is either
Riemannian or equipped with an arbitrary metric, $d$, 
compatible with its locally Euclidean topology.

Then, there exists a nonempty open conic neighborhood $\Omega \subset \tO $ of $(x_0,\xi_0)$;
 constants $\ve_0,\, \lambda_0>0$; 
a pseudodifferential 
cutoff $\,\Psi  \in \Psi^0(M)$, 
 elliptic at $(x_0,\xi_0)$ and microlocally supported in $  \Omega $; 
 a  real-valued, homogeneous function $S \in C^\infty(\tO)$; 
and homogeneous compound symbols $a, b \in S^0(\tO)$, so that 
for all $x,y \in \mbox{proj}_M(\Omega)$ with $d(x,y) \leq \varepsilon_0$, the Schwartz kernel of $ \Psi \Pi_{\lambda, \cbar} \Psi^*$ is
\begin{align}\label{eq:QCI_WeylLaw}
\begin{split}
& \Psi \Pi_{\lambda, \cbar} \Psi^* (x,y)\\
& \hskip 0.5in = \frac{1}{(2\pi)^n}\int\limits_{I(\lambda, \overline{c}) \cap \overline{p}(\Omega)} 
e^{i(S(x,\xi) - S(y,\xi))}b(x,\nabla_\xi S(x,\xi);\xi)a(\nabla_\xi S(y,\xi), y; \xi) \, d\xi+ \mathscr{R}(x,y,\lambda),
\end{split}
\end{align}
where, uniformly for all $\lambda \geq \lambda_0$ such that $\lambda \cbar \notin \Lambda$,
\[\sup\limits_{ \{x,y \in proj_M(\Omega) \, : \, d(x,y) \leq \epsilon_0 \}} \left|\mathscr R(x,y,\lambda)\right| =  \mathcal O(\lambda^{n-1}).\]
Above the diagonal of $M\times M$, $b(x,\nabla_\xi S(x,\xi);\xi)a(\nabla_\xi S(x,\xi), x; \xi) \equiv |\sigma(\Psi)(x,\xi)|^2$.
\end{thm}

\begin{rmk} Note that $\epsilon_0$ above is independent of $\lambda$, so that the representation
\eqref{eq:QCI_WeylLaw} is valid on a fixed tubular neighborhood of the diagonal in 
$\mbox{proj}_M(\Omega)\times \mbox{proj}_M(\Omega)$, uniformly in $\lambda\ge \lambda_0$.
\end{rmk}

\begin{rmk}
    Choosing $\overline{c}$ so that $\lambda \cbar \notin \Lambda$ for some sufficiently large values of $\lambda$ is always possible 
    since $\Lambda$ is discrete.  We emphasize that the remainder estimate on $\mathscr{R}$ does not depend on 
    $dist \left( \lambda \overline{c}, \Lambda \right).$
\end{rmk}

\begin{rmk} \label{rem:coord_free}
    It is intuitively clear that the coordinate-free form of the expression (\ref{eq:QCI_WeylLaw}) should have as its phase an 
    exponential-type map 
associated with the joint Hamiltonian flow, like in H\"ormander's expression (\ref{e:weyl_off-diagonal}) with the Hamiltonian flow 
generated by $|\xi|_{g}$, but we leave this for future exploration.
\end{rmk}

\begin{cor} \label{cor:main}
Suppose $(M,g)$  is Riemannian
 and $\left(-\Delta_g\right)^\frac12$ is QCI, with QCI system $\overline{P}$. 
 Then, with the hypotheses and notation of  Theorem \ref{thm:main},
if $x,y\in \mbox{proj}_M(\Omega)$ are such that $d_g(x,y) \le \frac{c}{\lambda} < inj(M)$, one has 
\begin{equation}\label{eqn:cor}
\Psi\Pi_{\lambda, \cbar}\Psi^*(x,y) = \frac{1}{(2\pi)^n} \int\limits_{I(\lambda, \overline{c}) \cap \overline{p}(\Omega)} 
e^{i(x-y)\cdot \nabla_xS(x;\xi) }\left|\sigma(\Psi)(y,\xi)\right|^2\,d\xi + \mathscr R(x,y,\lambda),
\end{equation}
with
\[\sup\limits_{d_g(x,y)\le \frac{c}{\lambda}}\left|\mathscr R(x,y,\lambda)\right| = \mathcal O(\lambda^{n-1}).\]
\end{cor}

\begin{rmk}
    Observe that the integrals (\ref{eq:QCI_WeylLaw}) and (\ref{eqn:cor}) could be re-expressed as integrals over 
    $\Omega \cap \overline{p}^{-1}(I(\lambda, \overline{c}))$ thanks to the fiber rank $n$ condition making $\overline{p}$ 
    a diffeomorphism on the fibers $T^*_x M.$  This would result in expressions closer to those in \cite{Hormander1968}, 
    covering the case of a single elliptic operator, which involved integrating over sub-level sets of the symbol.  
    However, this would result in more complicated phases and amplitudes in the referenced equations.
\end{rmk}

\subsection{Discussion of Theorem \ref{thm:main}}\label{sec:disc Thm 1}

\subsubsection{New information and implications}\label{sec: new info}

In the process of proving the main result, we spend a considerable portion of this article  deriving the asymptotics of a (microlocalized) 
smoothed joint spectral measure, $\rho(P_1 - c_1\lambda ) \circ \dots  \circ  \rho(P_n - c_n \lambda) $,  
encapsulated in Prop. \ref{prop:smooth_spec_measure_microlocal}, which is itself  possibly of independent interest. 
In fact, this result on the smoothed spectral measure immediately implies that joint eigenfunctions, 
microlocalized to $\Omega$ enjoy $L^2-L^{\infty}$ bounds 
which are $\mathcal{O}(1)$, a  considerable improvement over the standard H\"ormander bound of $\mathcal{O}(\lambda^{1/2})$.  
This recovers a result of Tacy \cite[Thm. 0.2]{Tacy2019} under the same fiber-rank condition, which generalized  previous work of 
Sarnak \cite{Sarnak_letter} and Marshall \cite{Marsh} on locally symmetric spaces.  We  also mention the improved sup-norm 
bounds of  Galkowski-Toth \cite{GalkowskiToth2020} under a Morse-type assumption which is less stringent than our fiber rank condition. 

It would be interesting to see whether our explicit description of the  kernel lends itself  to improved $L^p$ restriction estimates, 
along the lines of \cite{BGT07,Hu2009}, 
for  submanifolds that lie within $proj_M(\Omega)$.  
Such a question for restriction to curves in 2D was posed by Toth, with our result being consistent with
 \cite[Theorem 3 (i)]{Toth09}, given that we are using a microlocal cutoff. 

The content of the previous two paragraphs
demonstrate a well-known phenomenon in 2D,  most conveniently visualized on surfaces of revolution: 
the ``zonals", ``Gaussian beams", 
and the ``lone" $e^{i \theta \cdot \xi_{\theta}}$ Fourier element (see Section \ref{sec:exs}) are what contribute to the 
non-trivial sup-norm growth due to the $\mathbb S^1$ action.  
What is new here is that despite eliminating the $L^p$-norm saturating 
basis elements and the lone Fourier element, we still obtain a pointwise $\sim \lambda^n$ asymptotic in Theorem 1.

It should not be surprising that the fiber rank condition can be framed through the moment map $\overline{p}$; 
as mentioned in Remark \ref{rmk:def1.2}, it is equivalent with 
 $\pi: \Sigma \rightarrow M$ being a submersion for each level set $\Sigma$ of $\overline{p}$.
We emphasize that, at least to the knowledge of the authors, the only setting where $\Omega$ can be 
taken as all of $T^*M$ (and therefore $\pi|_{\Sigma}$ is a global diffeomorphism) is when $M$ is the flat torus 
(see \cite[p. 919]{GalkowskiToth2020} and   \cite[p. 17]{TZ01}).  In fact, by topological considerations \cite{TZ01}, 
it should be rare for a QCI system to have its $\Omega$ equal to all of $T^*M$, 
thus making  microlocalization by $\Psi$ reasonably motivated.  
For the torus one could choose the cutoff $\Psi$ not microlocally but via the spectral theorem and functional calculus.  
The resulting $\Psi$  allows us to count eigenvalues in cones contained inside of $I(\lambda, \overline{c})$, obtaining a 
non-microlocalized pointwise analogue of (\ref{eq:QCI_WeylLaw}) in that setting.   

We continue our discussion by revisiting a result of Duistermaat \cite[Thm. 1.1.1]{Duistermaat74} on quasimodes for semiclassical 
pseudodifferential operators $P(x, \frac{1}{\lambda} D_x)$.  Under the  version of our fiber rank assumption on the principal symbol of 
a single Hamiltonian $P$, i.e., $k=1$, Duistermaat showed that $P$ admits $O(\lambda^{-\infty})$ quasimodes in the form of WKB/
Lagrangian states.  Our Prop. \ref{prop:smooth_spec_measure_microlocal} gives a (microlocalized) QCI version of this result when 
$y$ is fixed and $x$ is allowed to vary: that is, the QCI system $\{P_1, \dots, P_n\}$ 
admit $\mathcal{O}(\lambda^{-\infty})$ (microlocalized) joint quasimodes in the form of WKB/Lagrangian states.

\subsubsection{Future questions}\label{sec:future}

One of the original motivations of this article was the connection between pointwise Weyl laws and zero sets of random linear 
combinations of joint eigenfunctions.  Statistics of the number of connected components of the zero sets of random functions is an area 
with a wealth of challenging and difficult questions, with a beautiful piece of science from Bogomolny-Schmitt \cite{BS02} addressing the 
connection between bond percolation and the complements of zero sets of random eigenfunctions  (that is, their \emph{nodal domains}).  
Two fundamental works of Nazarov-Sodin \cite{NS09, NS16} concerned combinations of eigenfunctions \emph{for the single operator of 
$-\Delta_g$} (and say sampled from the spectral interval $[\lambda, \lambda + M]$ where $M$ is constant but sufficiently large) and their 
nodal domains.  These breakthrough articles established, amongst other results, asymptotics for the expected number of nodal domains 
for large $\lambda$.  The work \cite{NS16} gave an important criterion for such an asymptotic to hold, specifically that the spectral 
projector's Schwartz kernel in (\ref{e:weyl_off-diagonal}) has its integration measure be absent of atoms (in the terminology of 
\cite{NS16}, its ``spectral measure" $\rho$ should have no atoms), which is clearly the case for the Laplacian. 
A variety of other interesting works were  motivated  by \cite{NS16}; 
see the survey article of Wigman \cite{Wig03} and the references therein.

In the case of ensembles of joint eigenfunctions, say whose eigenvalues are sampled from our conic regions $\overline{p}(\Omega)$, a 
description that could be used to verify Nazarov-Sodin's criterion has been absent.  A motivation for proving  Theorem 1 
was to give a non-trivial piece of evidence 
that the actual joint projector's kernel will satisfy this criterion.  
Our hope is to suggest new random ensembles of eigenfunctions, smaller 
than those considered by Nazarov-Sodin, but sharing their asymptotic.  
We believe that  QCI systems are a setting 
in which the exponential concentration of  nodal domain counting functions might be  possible; 
to prove see \cite{NS09, Roz17 } for the integrable cases 
of $\mathbb S^2$ and the torus.
  This an interesting question for further research. 

It would be remiss to not touch upon \cite{CdV79}.  We should mention that Colin de Verdi\`ere's ``integrated" Weyl law 
(\ref{eqn:integ_Weyl}), for the cones described in Theorem \ref{thm:main}, should follow from our asymptotic (\ref{eq:QCI_WeylLaw}) 
albeit with the appropriate Tauberian result.  Our pointwise Tauberian result Proposition \ref{taub_prop} is not robust enough to recover 
\cite[Theorem 1.4]{CdV79} as one needs to consider the integrated asymptotics coming from terms of the form $(I-\Psi) \Pi \Psi^*$ and 
$\Psi \Pi (I-\Psi^*)$.  This appears to require additional joint wave trace analysis, itself drawing upon the clean intersection calculus 
\cite{DuistermaatGuillemin1975}.  We did not pursue the necessary extensions and leave these problems for interested readers. 

\subsection{Outline of the proof}\label{sec:outline}

A major obstruction encountered in the QCI case, say when trying to follow most proofs of pointwise Weyl laws for single elliptic, 
self-adjoint operators, is the absence of 
amenable oscillatory expressions for the joint propagator $e^{it_1 P_1} \circ \dots \circ e^{i t_n P_n}$.  That is, it is difficult in general to 
control the multiple integrations needed when manipulating the various Schwartz kernels of the individual propagators. The key idea 
towards Theorem \ref{thm:main} is a microlocal normal form due to Colin de Verdi\`ere \cite[Thm. 2.1]{CdV79} which generalizes that 
of Duistermaat-H\"ormander \cite[Prop. 6.1.4]{DuistermaatHormander72} in the case of a single operator.  This normal form 
accomplishes a few things:
thanks to the power of the FIO calculus \cite{FIO1}, 
it allows us in Section \ref{sect:mnf}
to rely almost solely on explicit calculations for the joint propagator of the model QCI system 
$\{D_{x_1}, \dots, D_{x_n}\}$ on $\mathbb{R}^n$. 
The FIO calculus for operators associated to local canonical 
graphs supplies the generating functions $S$ for certain symplectomorphisms attached to $\{P_1, \dots, P_n\}$ appearing in 
the expression (\ref{eq:QCI_WeylLaw}); see Propositions \ref{p:DH_nf} and \ref{p:CdV_qbnf} for more details.  These 
advantages emerge in our calculation of the Schwartz kernel of the smoothed spectral measure 
$\rho(P_1 - c_1\lambda ) \circ \dots  \circ  \rho(P_n - c_n \lambda)$ which appears in Section \ref{sect:microlocal_asymps}.  

The passage from the smoothed spectral measure $\rho(P_1 - c_1\lambda ) \circ \dots  \circ  \rho(P_n - c_n \lambda) $ to the pointwise 
Weyl law (\ref{eq:QCI_WeylLaw}) comes, as usual, from an appropriate Tauberian theorem.  Our approach is to adapt a robust 
argument of Sogge \cite[Section 3.2]{SoggeBook2017}, for integrated Weyl laws in the single-operator case, to multiple operators.  Such 
a strategy was used previously in the work of Keeler \cite{Keeler2023} on pointwise Weyl laws for Laplacians in settings with non-
positive curvature.  Our relevant results, and the final steps of the proof of Theorem \ref{thm:main}, 
are described in Section \ref{sect:tauberian}.

\subsection{Notation} \label{sect:list}

To assist the reader, here is a list of notation used frequently in the paper.  

\begin{itemize}

\item $proj_M: T^*M \rightarrow M$ is the invariantly defined spatial projection, $(x;\xi) \mapsto x.$
\smallskip

\item $\Omega_0 \subset T^*M$ is the open, dense subset on which the $dp_1 \wedge \dots \wedge dp_n \neq 0$.  This appears in 
Definition \ref{defn:QCI}.  The set $\Omega$ in Theorem \ref{thm:main} appears in equation (\ref{eqn:Omega_set}, just before our 
Proposition \ref{prop:smooth_spec_measure_microlocal}. 
\smallskip

\item $\overline{P}=(P_1,\dots,P_n)$ will denote a QCI system on $M$, with principal symbols\linebreak
$\overline{p}=(p_1,\dots,p_n)$.  
\smallskip

\item The local model for a QCI system (see Prop. \ref{p:CdV_qbnf}), 
is $\overline{P^0}= \left(P^0_1,\dots,P^0_n\right)$ on $\R^n$, where $P_i^0 = D_{x_i}$ on $\mathbb{R}^n$, with  
principal symbols $p^0_i=\xi_i, \, i=1, \dots, n$, resp.
\smallskip

    \item $\Lambda \subset \mathbb{R}^n$ is the joint spectrum of the QCI system $\overline{P}$,  
whose elements are the joint eigenvalues, denoted $\overline{\lambda}_j =\left(\lambda_j^{(1)},\dots,\lambda_j^{(n)}\right)$.   
A corresponding orthonormal basis of eigenfunctions will be denoted $\varphi_{\overline{\lambda}_j}.$
\smallskip

 \item 
 For notational convenience, the Schwartz kernel of a linear operator $A$ will be denoted by either 
 $K_A(\cdot,\cdot)$ or $A(\cdot, \cdot)$
 (e.g., as in \eqref{eq:QCI_WeylLaw}).
 It should be clear from the context whether it needs to be treated as a distribution or is actually a smooth function.  
 \smallskip

 \item As in Sec. \ref{sect:setup_main}, $\overline{c} = (c_1, \dots, c_n) \in \mathbb S^{n-1}$ is a fixed unit vector, 
 all of whose components are non-zero. For $\lambda>0$, 
 $I(\lambda, \overline{c}) = \prod_{i=1}^n  [-|c_i| \lambda, |c_i| \lambda]$
is the spectral support of the orthonormal projection operator $ \Pi_{\lambda, \cbar }$ from \eqref{eqn:joint_spectral_function},
whose Schwartz kernel  is, as in the previous bullet point,  denoted $ \Pi_{\lambda, \cbar }(x,y)$.
 The vector  $\overline{c}$ will be chosen in Sec. \ref{sect:final_steps}
 \smallskip

    \item $\epsilon_0 > 0 $ is a fixed number connected to the joint Hamiltonian flow generated by the symbols $p_1, \dots, p_n$ and the 
    open set $U$ as in the symplectic normal form, Prop. \ref{p:DH_nf}.  
    It appears for the first time in equation (\ref{e:nested_flow_nbhds}) and is finally chosen  in equation (\ref{eqn:Omega_set}). 
    In terms of $\epsilon_0$, we  define
     $J(\epsilon_0):=\left(-\frac{\epsilon_0}{n+1},\frac{\epsilon_0}{n+1}\right)\subset\R$ and its $n$-fold Cartesian product, 
     $\overline{J}(\epsilon_0):=J(\epsilon_0)^n$.
 \smallskip
    
    \item $\{ W_{k \epsilon_0/(n+1)} \}_{k=1}^{n}$ is a nested chain of conic open subsets of another open set $U$ appearing in
Prop. \ref{p:DH_nf} and which determines the set $\Omega$ appearing in Thm. \ref{thm:main}.  
This chain first appears in equation (\ref{e:nested_flow_nbhds}).
 \smallskip
 
    \item If $U\subset T^*M\setminus 0$ is open and conic, $Q_U$ 
    will often be used to denote an element of $\Psi^0(M)$ 
    (with $q_U=\sigma(Q_U)\in S^0$ its principal symbol) 
    having microlocal support 
    contained in $U$, i.e., such that $\mbox{ supp } Q_U \subset U$, in the sense that $Q_Uf\in C^\infty$ 
    whenever $WF(f)\cap{U}=\emptyset$.
     If $U'\subset U$, then $Q_{U',U}$ will denote a $Q_U$ 
    for which $Q_{U',U}\equiv I$ on $U'$, i.e., $Q_{U',U}-I$ is smoothing on $f$ with $WF(f)\subset U'$.
  
\end{itemize}

\subsection{Acknowledgements}\label{subsec:ack}

We would like to thank Matthew Blair, Iosif Polterovich, Zeev Rudnick, Melissa Tacy and Jared Wunsch
for generously making helpful comments on a draft version of this paper. AG was partially supported by US 
NSF award DMS-2204943, SE by the NSERC Discovery Grant Program, 
and BK by a CRM-ISM Fellowship and AARMS Fellowship during the writing of this article.


\section{Examples} \label{sec:exs}


\subsection{Flat tori}\label{subsec:torus} Let $M=\mathbb T^n=\R^n/{(2\pi \mathbb Z)^n}$ be the standard $n$-dimensional flat torus,
with coordinates $x=(x_1,\dots,x_n)\in [0,2\pi)^n$ and dual coordinates $\xi=(\xi_1,\dots,\xi_n)$ on the fibers of $T^*M$.
Set $\mathcal P=\{D_1,\dots,D_n\}$, where $D_j=\frac1{i}\partial_{x_j},\, 1\le j\le n$.
Then $\mathcal P$ is a QCI system, for which we may take $\Omega_0$ to be all of $T^*M\setminus 0$;
since $\sigma(D_j)=\xi_j$, this satisfies the full fiber rank condition everywhere as well, so may take $\tO=\Omega_0$.
The moment map is $\overline{p}(x,\xi)=(\xi_1,\dots,\xi_n)$,
and $ \Gamma = \overline{p}(T^*M\setminus 0) =\R^n \setminus 0$,
so that in order for  the set $\mathcal C$ mentioned in Sec. \ref{sec: QCI} to be disjoint from $\Gamma$, 
$\mathcal C$ must be empty.
The joint spectrum of $\mathcal P$ is, of course,   $\Lambda=\mathbb Z^n$. 
By   Strichartz  \cite{Strichartz1972},  taking $\Psi \in \Psi^0$ to be translation invariant
results in a  smooth joint spectral cutoff, and  by translation-invariance has  spatial support 
which is  all of $\mathbb{T}^n$.
\medskip

\subsection{Ellipsoids}\label{subsec:ellipsoid}
Now take $(M,g)$ to be an ellipsoid of revolution about the $[-a,a]\subset x_1$-axis in $\mathbb{R}^3$. 
Let  $N=(-a,0,0),\, S=(a,0,0)$ be the poles; $\gamma=M\cap\{x_1=0\}$ the equator;  
$s$ and $\theta$  the arc-length  (from $N$ to $S$ along the ellipse)    and  angular  coordinates, resp., 
with dualn coordinates $\xi=(\sigma, \Theta)$.
Let $\mathcal P=\{P_1,P_2\}$ with
$P_1=(-\Delta_g)^\frac12$ and $P_2 = D_{\theta}$, whose principal symbols are $p_1=|\xi|_g,\, p_2=\Theta$, resp.
 Then $\mathcal P$ is  a QCI system  on $M$, satisfying the fiber rank condition on
$$\Omega_0=\tO=T^* M \backslash\left(  T_N^*M \cup T_S^*M \cup  T^*\gamma\right),$$
where $T^*\gamma=\{ (x,\xi)\in T^*M|_\gamma\, :\, \sigma=0\}$. 
(Note that each of $T^*M|_N,\,  T^*M|_S$ and $  T^*\gamma$ is a codimension  
two conic  submanifold of $T^*M$.)
Next, notice that the set $F$ of critical points of the moment map $\overline{p}(x,\xi)=\left(|\xi|_g, \Theta\right)$ is only $ T^* \gamma$.  
Colin de Verdi\`ere's counting result applies to this $M$ for $k=n=2$, 
with the conic set $\mathcal{C} = \Omega$; 
it follows that 
\begin{equation}
    \mbox{Vol} \left( \, \overline{p}^{-1} \left\{ \mathcal{C} \cap B(0, \lambda) \right\} \, \right) = c_0 \lambda^2 + \mathcal{O}(\lambda)
\end{equation} 
where $c_0 >0$ and depends on $\mathcal{C} \cap  S^{*}M.$  Hence, Theorem \ref{thm:main} in the current paper is 
an off-diagonal pointwise analogue of Colin de Verdi\`ere's 
 estimate, for a restricted class of cones.  Repeating the remark made just before Theorem \ref{thm:main}, in the context of H\"ormander's 
result (\ref{e:weyl_off-diagonal}), our main result is a natural QCI-extension but with the drawback of being stated in local coordinates.  

\medskip

\subsection{Surfaces of revolution}\label{subsec:surf rev}

Suppose that $g$ is a metric on $\mathbb S^2$ for which there exists an effective action of $\mathbb S^1$,
with fixed points denoted by $N$ and $S$.
Let $\mathcal P=\{P_1,P_2\}$ with $P_1 = \sqrt{-\Delta_g}$ and $P_2 = \partial_{\theta}$.  
We examine this system over the three standard coordinate charts making up an atlas,
covering $\mathbb S^2\setminus \{N,S\}$ and neighborhoods of $N,\, S$, resp..

We take $(\sigma, \theta) \in (0,L) \times [0,2 \pi]$ to be geodesic polar coordinates centered at $N$, 
but not including $N$ or $S$, with $\theta=0$ some fixed  meridian $\gamma$ from $N$ to $S$.  
The dual coordinates we denote $(\Sigma,\Theta)$.
On this  chart, covering  $\mathbb S^2\setminus\{N,\, S\}$, the metric can  be written as
$g = d \sigma^2 + a(\sigma)^2 d \theta^2$ where $a: [0,L] \rightarrow \mathbb{R}^+$ satisfies $a(0)=a(L)=0$ and 
$a^{(2k)}(0)=a^{(2k)}(L) = -1$, for some integer $k \ge 1$.  We assume $a$ has finitely many 
critical points, $\sigma_j \in (0,L)$, with $j=1,...,M$.  In this chart, $p_1 = \sqrt{\Sigma^2 + \frac{\Theta^2}{a^2(\sigma)}}$ 
and $p_2 = \Theta$, from which one sees immediately that $\nabla_{\Sigma, \Theta} p_1$ and $\nabla_{\Sigma, \Theta} p_2$ 
are linearly independent over this chart if and only if $\Sigma \neq 0.$  Therefore, $\tO$, and thus any $\Omega$, 
must exclude $\{ \Sigma = 0 \}$.

On the other hand, using Riemannian normal coordinates $x_1,x_2$ \`a la the exponential map in the charts containing $N$ and $S$, 
one has $p_1 = \Op \left( \sqrt{\sum_{i,j} g^{ij}(x) \xi_i \xi_j} \right)$ and $p_2 = \Op( x_2 \xi_1 - x_1 \xi_2)$.  It is well-known that $P_1$ 
and $P_2$ form a QCI system.  

Note that $d_{\xi}p_2$ equals $(0,0)$ at both $N$ and $S$, showing these points must be excluded.  Furthermore, a calculation shows that $H_{p_1}$ and $H_{p_2}$  are parallel only on $\bigcup_{j=1}^M T^* \gamma_j$,
where  $\gamma_j = \{ \sigma = \sigma_j \} \subset \mathbb S^2$.
Now, let $\tO = T^* ( \mathbb S^2 \backslash \{N,S \}) \backslash \bigcup_{j=1}^M T^* \gamma_j$,     
Therefore, the fiber rank condition fails on these two coordinate charts only above 
$N,S$ and thus any projection of $\Omega$ must exclude these points.  
Hence, Theorem \ref{thm:main} applies for any choice of $\Omega$ which excludes $T_N \mathbb S^2, T_S \mathbb S^2$, 
and $\{\Sigma=0\}.$

\subsection{Liouville tori}\label{subsec:liou torii} In a slightly different vein, consider the two-torus $M = \mathbb{R}^2 / \mathbb{Z}^2$ 
but now with two smooth, positive, and periodic functions $U_1, U_2: \mathbb{R}/\mathbb{Z} \rightarrow \mathbb{R}^+$ 
where we assume that $U_1(x_1) - U_2(x_2) > 0$.  
Take the metric $g = \left(U_1(x_1) - U_2(x_2) \right) \left( dx_1^2 + dx_2^2 \right)$ with coordinates 
$(x_1,x_2) \in [0,1] \times [0,1].$  Take $p_1 = \sqrt{\frac{\xi_1^2 + \xi_2^2}{U_1(x_1) - U_2(x_2)}}$ and 
$p_2 = \sqrt{\frac{U_2(x_2)}{U_1(x_1) - U_2(x_2)} \xi_1^2 + \frac{U_1(x_1)}{U_1(x_1) - U_2(x_2)} \xi_2^2}$.  
One knows that $P_1:=\Op(p_1), P_2:=\Op(p_2)$ form a QCI system with $P_1$ 
being the so-called Liouville Laplacian \cite{TZ01}.

Define $\Omega_1 = \{ (x; \xi) \in T^*M \, : \, \xi_1, \xi_2 \neq 0 \}.$  
The linear independence of the differentials $dp_1, dp_2$ on $\Omega$ is immediate.  Set $m(x;\xi) = 
\left|\begin{matrix}
    \partial_{x_1} p_1 & \partial_{x_1} p_2 \\
    \partial_{x_2} p_1 & \partial_{x_12} p_2
    \end{matrix} \right|
$, 
whose nonvanishing at a point $(x;\xi)$ implies that 
$\xi \cdot \partial_{\xi} \notin \mbox{span}\{H_{p_1}, H_{p_2} \}$ at $(x;\xi)$, 
which in turn shows that the hypotheses of Prop. \ref{p:DH_nf} are satisfied there.  
Setting 
$\Omega_2 = \{ (x; \xi) \in T^*M \, : \, m(x;\xi) \neq 0 \}$, then   
thanks to the assumption that $U_1 > U_2$, we see the fiber rank condition also holds on $\Omega := \Omega_1 \cap \Omega_2$.

\subsection{Quantum asymmetric top}\label{subsec:asymm top}
Toth \cite{TothVarious} studied a quantum version of the asymmetric top. This is a QCI system on $\mathbb S^2$, consisting of $P_1=\sqrt{-\Delta_{\mathbb S^2}}$, the Laplace-Beltrami operator for the standard (round) metric, and $P_2=\sqrt{-\Delta_{SO(3)}'}$, 
Laplace-Beltrami operator for another metric, induced from the action of $SO(3)$.  From \cite[Eqn. (2.21),(2.22)]{TothVarious},
the squares of the symbols of these two operators can be written as quadratic forms with the same major and minor axes.
Simplifying the coordinates and  notation from \cite{TothVarious}, they are conformal to 
$$p_1(u,v,\xi,\eta)=a(u)\xi^2+a(v)\eta^2,\quad p_2(u,v,\xi,\eta)=4va(u)\xi^2+4ua(v)\eta^2.$$
These have collinear gradients only on the axes; thus, we may take $\tO=\{\xi\ne0,\, \eta\ne 0\}$.


\section{The Joint Projector and the Joint Wave Kernel} \label{sect:joint_proj_wave}

In this section, we review how to relate the joint spectral projection kernel to the joint wave kernel. 
For  $\overline{t} = (t_1, \dots, t_n) \in \mathbb{R}^n$, consider the Schwartz kernel of the joint propagator, 
\[e^{i\tbar \cdot \overline P}(x,y) := e^{i t_1 P_1} \circ \dots \circ e^{i t_n P_n}(x,y).\]
Recalling the image of the moment map, $\Gamma$, from \eqref{e:image_mm},
let $\cbar \in \Gamma\cap \mathbb S^{n-1}$ (to be chosen later) with all  $c_i\ne 0$.
 Let $\Psi$ be an order-zero pseudodifferential operator. Then, we can rewrite the microlocalized joint spectral function as 
\begin{equation} \label{eqn:micro_joint_proj_spectral}
\Psi\Pi_{\lambda, \cbar }\Psi^*(x,y)= \sum\limits_{j=0}^\infty 
 \mathds 1_{I(\lambda, \overline{c})}\left(\overline{\lambda}_j \right)
\left( \Psi\varphi_j \right)(x) \left( \overline{\Psi\varphi_j(y)} \right).
\end{equation}

The Fourier transform of the characteristic function $\mathds 1_{[-\lambda,\lambda]}$ is 
\[\int\limits_{-\infty}^\infty e^{-it\mu}\mathds 1_{[-\lambda,\lambda]}(\mu)\,d\mu = 
\int\limits_{-\lambda}^\lambda e^{-it\mu}\,d\mu = \frac{2\sin (\lambda t)}{t},\]  
while the inverse Fourier (cosine) transform of $\sin(\lambda \cdot)/\cdot$ at $\lambda$ equals $1/2$.  
To avoid this issue, we demand that $\overline{c}$ be chosen so that $\lambda \cbar$ is 
not in the joint spectrum $\Lambda$ for sufficiently large $\lambda$. 
It is possible to do this, since $\Lambda$ is discrete: given   
$\overline{c} \in \mathbb S^{n-1}$, there exists $R_0(\overline{c})>0$ and a corresponding 
$\{\lambda(l)\}_{l\in\mathbb N} \subset B(0, R_0)^{\complement}$ such that $\lambda(l) \overline{c} \notin \Lambda$. 
Such a choice of spectral parameters is also made in \cite[p. 53]{SoggeBook2017}  
Therefore, through this Fourier inversion, we can  rewrite $\Psi\Pi_{\lambda, \cbar}\Psi^*(x,y)$ as 
\[\Psi\Pi_{\lambda, \cbar}\Psi^*(x,y) = \frac{1}{\pi^n}\sum\limits_{j=0}^\infty \lp\int\limits_{\R^n} e^{i \lambda_j^{(1)} t_1/|c_1|}\dotsm 
e^{i\lambda_j^{(n)} t_n/|c_n|}\frac{\sin(\lambda t_1)}{t_1}\dotsm \frac{\sin(\lambda t_n)}{t_n}\,dt\rp \Psi\varphi_j(x)
\overline{\Psi\varphi_j(y)}.\]
Changing variables via $t_k\mapsto |c_k|t_k$, we obtain 
\[\Psi\Pi_{\lambda, \cbar}\Psi^*(x,y) = \frac{1}{\pi^n} \sum\limits_{j=0}^\infty\int\limits_{\R^n} e^{i\tbar\cdot\lambdabar_j}
\frac{\sin(\lambda t_1|c_1|)}{t_1}\dotsm \frac{\sin(\lambda t_n|c_n|)}{t_n}\Psi\varphi_j(x)\overline{\Psi\varphi_j(y)}\,dt.
\] 
Since we are working at the level of Schwartz kernels, we can  interchange the sum and the integral to obtain 
\[\Psi\Pi_{\lambda, \cbar}\Psi^*(x,y) = \left( \frac{1}{\pi^n}\int\limits_{\R^n}\frac{\sin(\lambda t_1|c_1|)}{t_1}\dotsm 
\frac{\sin(\lambda t_n|c_n|)}{t_n}\lp \Psi\circ e^{it_1 P_1}\circ\dotsm \circ e^{it_nP_n}\circ\Psi^*\rp \, d \overline{t} \right) (x,y).\]  
It is on this expression for (\ref{eqn:micro_joint_proj_spectral}) that we perform the Tauberian analysis  in 
Section \ref{sect:tauberian}.

Now let $\overline{\mu} \in \mathbb{R}^n$.  A similar application of spectral calculus allows  us to  rewrite the spectral measure kernel 
(in the sense of distributions) as
\begin{align*}
\sum\limits_{j=0}^\infty \delta(\overline{\mu}-\lambdabar_j)\Psi\varphi_j(x)\overline{\Psi\varphi_j(y)} & =\frac{1}{(2\pi)^n}
\int\limits_{\R^n}e^{i\tbar\cdot\overline{\mu}}\lp\Psi\circ e^{-it_1P_1}\circ\dotsm\circ e^{-it_nP_n}\circ \Psi^*\rp(x,y)\,d\tbar\\
& = \left( \frac{1}{(2\pi)^n}\int\limits_{\R^n} \Psi e^{i\tbar\cdot(\overline{\mu}-\overline P)}\Psi^*\,d \overline{t} \right)(x,y).
\end{align*}
Integrating  in $\overline{\mu}$ over $I(\lambda, \overline{c})$ recovers the spectral projector 
function in (\ref{eqn:micro_joint_proj_spectral}).

Since the formula for the microlocalized spectral measure only holds in the sense of Schwartz kernels, 
as the integral in $\overline{t}$ is not 
absolutely convergent, we regularize it by introducing a Schwartz function $\rho\in \mathscr S(\R)$, even and  
such that $\supp \wh\rho \subset [-\delta_0,\delta_0]$ 
with $\wh\rho(s) \equiv 1 $ for $|s| < \delta_0/2$.  (The value of $\delta_0$ will be specified in Section \ref{sect:final_steps}.)  
Then, by 
properties of the Fourier transform and the spectral calculus, 
we have that 
\begin{align*}
\Psi\rho(\mu_1-P_1)\dotsm\rho(\mu_n-P_n)\Psi^*(x,y)&= \left( \frac{1}{(2\pi)^n}\int\limits_{\R^n}\Psi 
e^{i\tbar\cdot(\overline{\mu}-\Pbar)}\Psi^*\wh\rho(t_1)\dotsm\wh\rho(t_n)\,d\overline{t} \right)(x,y)
\end{align*}
The goal of the next  three sections is to analyze the asymptotic behavior of this {\it microlocalized smoothed spectral measure} using 
microlocal normal forms.  We then use a Tauberian statement to relate this smoothed measure to (\ref{eqn:micro_joint_proj_spectral}), 
thus proving Theorem \ref{thm:main}.


\section{Microlocal Normal Forms for Joint Propagators} \label{sect:mnf}

One  of the key facts used in the analysis of the microlocalized joint spectral function is the fact that we can microlocally transform 
the operators $\{P_1,\dotsc,P_n\}$ into simply $\{D_{x_1},\dotsc,D_{x_n}\}$ acting on open sets in Euclidean space. This is a 
consequence of the following well-known result in symplectic geometry (that is, the homogeneous Darboux theorem), which states that 
commuting Hamiltonians can be locally ``straightened out" into a convenient normal form.  
Recall the notations for $\overline{P}$ and $\overline{P^0}$ set out in Section \ref{sect:list}.

\begin{prop}(\cite[Prop. 6.1.3]{DuistermaatHormander72}) \label{p:DH_nf}
Let $k \leq n$.  Let $p_1,\dotsc,p_k$ be real-valued $C^\infty$ functions, homogeneous of degree 1 on a conic neighborhood of 
$(x_0,\xi^0)\in T^*M\setminus 0$. Then, for the existence of a homogeneous canonical transformation 
\begin{equation} \label{eqn:DH_symplecto}
    \chi(x,\xi) = (X_1(x,\xi),\dotsc,X_n(x,\xi);\, \Xi_1(x,\xi),\dotsc,\Xi_n(x,\xi))
\end{equation}
from a conic neighborhood $U$ of $(x_0,\xi^0)$ to a conic neighborhood $V$ of $(0,\Xi^0)\in T^*\R^n\setminus 0$ such that 
\begin{equation}\label{eq:DH_cara}
p_j(x,\xi) = \Xi_j(x,\xi), \quad j=1,\dotsc,k,
\end{equation}
it is both necessary and sufficient to have that
\begin{itemize}
\item[(i)] the Poisson brackets satisfy $\{p_i,p_j\} = 0$  for $i,j = 1,2,\dotsc, k$ in a neighborhood of $(x_0,\xi^0)$; and
\smallskip 

\item[(ii)] $H_{p_1}(x_0,\xi^0),\dotsc, H_{p_k}(x_0,\xi^0)$ and the radial vector $Y(x_0;\xi^0):=\sum_1^n\xi_i^0\partial_{\xi_i}$
are linearly independent. 
\end{itemize}
\end{prop}

The next proposition states the quantized version of the preceding result. This proposition appears in \cite{CdV79}, but we present a 
somewhat restructured proof here because it provides explicit constructions of the amplitude functions 
which appear in the statement of \thmref{thm:main}.

\begin{prop}(\cite[Thm. 2.1]{CdV79}) \label{p:CdV_qbnf}
Let $\overline{P}=\left(P_1, \dots, P_n \right)$ be a QCI system on $\Omega_0\subset T^*M\setminus 0$ and  
$(x_0;\xi^0) \in \Omega_0$. 
Then there exists a conic neighborhood $U$ of $(x_0,\xi^0)$ and a homogeneous canonical transformation 
$\chi: U \rightarrow T^* \mathbb{R}^n\setminus 0$ such that the conic neighborhood $V:=\chi(U)$ 
of $\chi(x_0;\xi^0)$ satisfies the following 
properties:  For any $Q_U\in \Psi^0(M), Q_V \in \Psi^0(\R^n)$ such that $\mbox{supp } q_U \subset U$ 
and $\mbox{supp } q_V \subset V$, we have that
\begin{enumerate}
\item There exists a closed conic subset $\mathcal{C}$ of the graph of $\chi$ and Fourier integral operators 
$A \in I^0(\mathbb{R}^n, M; \mathcal{C})$, $B \in I^0(M, \mathbb{R}^n; (\mathcal{C})^{-1})$ satisfying 
$(BAQ_U)- Q_U \in \Psi^{-\infty}(M)$ and $(ABQ_V) - Q_V \in \Psi^{-\infty}(\mathbb{R}^n)$.
\smallskip

\item Furthermore,  $(BP_i^0A Q_U) - P_iQ_U\in \Psi^{-\infty}(M)$
 and $P_i^0Q_V - (AP_iBQ_V) \in \Psi^{-\infty}(\mathbb{R}^n)$ for all $i = 1, \dots, n$.
\end{enumerate}
\end{prop}

\begin{rmk}
The terminology of ``quantum Birkhoff normal form" for results similar to that of Proposition \ref{p:CdV_qbnf} is standard in the 
semiclassical analysis literature.  While this proposition certainly functions as the quantization of a Birkhoff normal form, 
we are workinging with  homogeneous microlocal analysis and so will  use the standard terminology
of ``microlocal normal form"  from that literature.
\end{rmk}

\begin{proof} To prove Prop. \ref{p:CdV_qbnf},
we follow the broad strokes of the proof of \cite[Prop. 6.1.4]{DuistermaatHormander72}.  
First, we seek to apply Prop. \ref{p:DH_nf} to the principal symbols $p_1,\dotsc,p_n$.  
Note that condition (i) is immediate from Def. \ref{defn:QCI} and the  pseudodifferential calculus.  
For condition (ii), note that the linear independence of the set $\{H_{p_1}, \dots, H_{p_n}\}$ follows from the non-vanishing of 
$d{p_1} \wedge \dots \wedge d{p_n}$ on $\Omega_0$.  Hence, it remains to show that the radial vector $Y$ is transversal 
to $\mbox{span} \{H_{p_1}, \dots, H_{p_n}\}$ at $(x_0;\xi^0)$. 
Setting $\gamma_j(t) = \exp(tH_{p_j})(x_0;\xi^0)$, for fixed $j$ we have
\begin{equation*}
\frac{d}{dt}\Big(\sum_{i=1}^n p_i^2\left(\gamma_j(t)\right)
\Big)\big|_{t=0} = 
\sum_i 2 p_i(x_0;\xi^0) d{p_i}(H_{p_j}) = \sum_i 2p_i(x_0;\xi^0) \omega(H_{p_i}, H_{p_j}) = 0
\end{equation*}
since the symbols of the QCI system $\overline{P}$ Poisson commute.  
Now let $c = \sum_i p_i^2(x_0;\xi^0)$, which is $>0$ by the ellipticity of $\sum_i P_i^2$.
Since $j$ was arbitrary, it follows that $\mbox{span} \{H_{p_1}, \dots, H_{p_n}\} \subset T_{(x_0;\xi_0)}( \{ \sum_i p_i^2 = c \} )$.  
Now, as $\sum_i p_i^2$ is homogeneous of degree 2, by Euler the radial direction $Y$ satisfies $Y\big( \sum_i p_i^2\big) 
= 2\sum_i p_i^2= 2c > 0$,  
showing that condition (ii) is also satisfied, and the assumptions of Prop. \ref{p:DH_nf} hold for $k=n$.
\smallskip

By Prop. \ref{p:DH_nf}, there exists an open conic neighborhood  $\tilde{U}$ of $(x_0;\xi^0)$, and  a canonical transformation 
$\chi: \tilde{U} \to \tilde{V} \subset T^*\R^n\setminus 0$ such that \eqref{eq:DH_cara} holds. 
Let $\mathcal{C}\subseteq Gr(\chi)$ be any open conic subset of the graph of $\chi$, and choose any 
$A_1\in I^0(\mathbb{R}^n, M;Gr(\chi))$ 
which is supported on $\mathcal{C}$, elliptic on $\tilde{U}$ and whose principal symbol $a_1 \equiv 1$ 
on an open conic set $U\subset\tilde{U}$ containing  $(x_0;\xi^0)$.
By slight abuse of notation, we will write $A_1\in I^0(\mathbb{R}^n, M;\mathcal{C})$ and similarly for related operators.
Since $A_1$ is elliptic on $U$, it has a microlocal parametrix $B_1\in I^0(M, \mathbb{R}^n;(\mathcal{C})^{-1})$, 
so that, for any microlocal cutoff $Q_{U}\in\Psi^0(M)$  supported in  $U$,  $(A_1B_1Q_{U})- Q_{U} \in\Psi^{-\infty}(M)$.
\smallskip

{At this point, for notational consistency with \cite{DuistermaatHormander72}, for the remainder of the proof we will use
use $(x,\xi)$ to denote the standard canonical coordinates on $T^*\R^n$.}
For each $j$, Egorov's Theorem and the transverse intersection calculus  give us 
\begin{equation*}
    P_j B_1 - B_1(D_{x_j} + \sum_{l=0}^{\infty} Op(q_j^{(-l)})) \in I^{-\infty}(M , \mathbb{R}^n; \mathcal{C})
\end{equation*} 
for some  symbol expansion $\sum_l q_j^{(-l)} \in S^0$.  
Our goal is to find a single elliptic pseudodifferential operator $A_2\in\Psi(\R^n)$, independent of $j$, 
elliptic on $\chi(U)$ and such that 
\begin{equation} \label{e:pseudor_equivalence}
(D_{x_j} + \sum_{l=0}^{\infty} Op(q_j^{(-l)}))A_2 - A_2 (D_{x_j}) \in \Psi^{-\infty}(\R^n)
\end{equation}
for all $j=1,\dots,n$. 
One can rewrite the relation (\ref{e:pseudor_equivalence}) as 
\begin{equation*}
[D_{x_j}, A_2] + \sum_{l=0}^{\infty} Op(q_j^{(-l)})A_2 \in \Psi^{-\infty}(\R^n), \, \,  \,  1 \leq j \leq n,
\end{equation*}
with this formulation allowing us to apply the procedure of  \cite[Prop. 6.1.4]{DuistermaatHormander72} whilst 
keeping track of a useful piece of information that we describe now.    
We will then set $A:=A_2A_1$.

 The operator $A_2\in\Psi^0(\R^n)$ will be constructed  in the form
\begin{equation} \label{e:A_2}
A_2 := A_2^{(0)} \prod_{k =1}^\infty (1 + C_{-k}),
\end{equation} 
where $C_{-l} \in \Psi^{-l}(\R^n)$ for each $l$. The right side is interpreted as an asymptotic limit in $\Psi^0(\R^n)$,
as $N\to\infty$,
of $A_2^{(N)} := A_2^{(0)} \prod_{k =1}^N (1 + C_{-k})$, 
constructed such that for each $1\le j\le n$,
\begin{equation*}
[D_{x_j},A_2^{(N)}] + \sum_{l=0}^{\infty} Op(q_j^{(-l)}))A_2^{(N)}
\in \Psi^{-(N+1)}(\R^n).
\end{equation*}
We start by finding the leading term, a microlocally elliptic $A_2^{(0)} \in \Psi^0(\R^n)$ such that 
\begin{equation*}
[D_{x_j}, A_2^{(0)}] + \sum_{l=0}^{\infty} Op(q_j^{(-l)})A_2^{(0)} \in \Psi^{-1},  \, \,  \,  1 \leq j \leq n.
\end{equation*}
Let $a_2^{(0)}$  be its principal symbol.  We will solve the system 
\begin{equation*}
\begin{cases}
\frac{1}{i} \partial_{x_1} a_2^{(0)} + q_{1}^{(0)} a_2^{(0)} &= 0 \\
\dots \\
\frac{1}{i} \partial_{x_n} a_2^{(0)} + q_{n}^{(0)} a_2^{(0)} &= 0.
\end{cases} 
\end{equation*}
After a possible translation to the origin in the $x$ coordinates,
it can be assumed that we are working near the basepoint $x=0\in \chi(U)$,
and we augment the above system 
with initial conditions $(a_2^{(0)})_{\upharpoonright x=0} = 1$.  
A calculation below that in more general form will lead to the relation (\ref{e:quantum_comm_relations})
 implies that 
\begin{equation} \label{e:derivative_relations}
\partial_{x_k} q_{j}^{(0)} = \partial_{x_j} q_{k}^{(0)},  \, \,  \,  1 \leq j,k \leq n,
\end{equation}
so that any one  of the $q_{j}^{(0)}$  determines all the other symbols $q_{k}^{(0)}$. 
It follows that, on a conic open subset $V \subset \chi(U)$, there is a solution of the form
\begin{equation*}
a_2^{(0)}(x;\xi) = \exp(-i \int_0^{x_1} q_1^{(0)}(t_1,x';\xi) \, dt_1).
\end{equation*}
One can verify that $a_2^{(0)} \in S^0$, equals  $1$ at $x=0$, and thus (possibly shrinking $V$) is elliptic on $V$. 
Let $A_2^{(0)} = \Op(a_2^{(0)})$ and  $B_2^{(0)}$ be a microlocal  parametrix  of $A_2^{(0)}$  on $V$.  
By $\Psi$DO calculus, since $ord\left(q_j^{(-l)}\right)=-l\le -1$ for $l\ge 1$, we have 
\[[D_{x_j}, A_2^{(0)}] + \sum_{l=0}^{\infty} Op(q_j^{(-l)})A_2^{(0)} =: \tilde{R}_j^{(-1)} \in \Psi^{-1}(\R^n),\,  \,  1 \leq j \leq n.\] 
Defining $R_j^{(-1)} = B_2^{(0)} \tilde{R}_j^{(-1)}$, 
the principal symbol $r_1^{(-1)}$ of (say) $R_1^{(-1)}$ depends only on $q_1^{(0)}$.  
Furthermore, by the calculations below
(leading from the commutation relations (\ref{e:quantum_comm_relations}) 
leading to the derivative relations \eqref{e: new 4.5}), 
the same relation amongst the spatial derivatives of 
the $r_{j}^{(-1)}$ hold as for the $r_{j}^{(0)}$ as in (\ref{e:derivative_relations}), 
so that $r_j^{(-1)}$ also depends only on $r_1^{(-1)}$ for any $j=1,\dots, n$.  

The derivation of $A_2^{(0}$ is the initial step in the recursive construction of the $A_2^{(N)}$, which we 
base on an argument in  \cite{CdV79}.  Note that $B_1[P_j, P_k] A_1 Q_U \equiv 0$.  
Suppose that for $N\ge 0$, we have an 
$A_2^{(N)}\in\Psi^0(\R^n)$ that is elliptic on $V$ and with microlocal parametrix $B_2^{(N)}$ satisfying the property
that,  for all $1\le j\le n$,
\begin{equation*}
    A_2^{(N)} \, A_1 \, P_j \, B_1 \, B_2^{(N)} \equiv D_{x_j} + R_j^{-(N+1)} (\mbox{mod} \, \Psi^{-\infty}),  \, \,  \,  1 \leq j \leq n,
\end{equation*}
for a $\Psi$DO  $R_j^{-(N+1)}$ of order $-(N+1)$ with principal symbol  satisfying $r_j^{-(N+1)}$.
Then we can calculate as follows: Recalling that $A_1B_1\equiv I \mod \Psi^{-\infty}(\R^n)$, we have
$$P_j B_1 A_1 P_k \equiv P_j P_k \mod \Psi^{-\infty},$$
and hence
\begin{align} \label{e:quantum_comm_relations}
\nonumber  A_2^{(N)} \, A_1[P_j,P_k]B_1 \, B_2^{(N)} & \equiv \left[ D_{x_j} + R_j^{-(N+1)}, D_{x_k} + R_k^{-(N+1)} \right]  \\
& \equiv 0 \, (\mbox{mod} \, \Psi^{-\infty}),  \, \,  \,  1 \leq j,k \leq n.
\end{align}
Taking principal symbols and grouping together the like terms,
\eqref{e:quantum_comm_relations}
implies  relations among the corresponding principal symbols, namely
\begin{equation}\label{e: new 4.5}
   \partial_{x_k} r_j^{-(N+1)} = \partial_{x_j} r_k^{-(N+1)},  \, \,  \,  1 \leq j,k \leq n.
\end{equation}

Having established \eqref{e:quantum_comm_relations} and \eqref{e: new 4.5}, the case of $N=0$ finishes the
construction of $A_2^{(0)}$. We now turn to the iterative construction of the $A_2^{(N)},\, N\ge 1$.
 
 The relations \eqref{e: new 4.5} just derived will allow us to solve a system of non-homogeneous transport equations,
applying the remaining steps of Duistermaat-H\"ormander's argument but with a slight variation, 
whose purpose is to continue to  obtain  
simple systems of linear ODEs which are explicitly  solvable by integrating factors.
This will lend itself  to the existence of single pseudodifferential operator 
$A_2$ such that \eqref{e:pseudor_equivalence} holds for all $P_j,\, 1\le j\le n$.  
Recall that the Ansatz is to find, via induction, 
$A_2^{(N)} = A_2^{(0)}(1 + C_{-1}) \dots (1 + C_{-N})$, where $C_{-l} \in \Psi^{-l}$ for each $l$, such that for each $1\le j\le n$,
\begin{equation*}
(D_{x_j} + \sum_{l=0}^{\infty} Op(q_j^{(-l)}))A_2^{(N)} - A_2^{(N)} D_{x_j} \in \Psi^{-(N+1)}(\R^n).
\end{equation*}
After findng $A_2^{(N-1)}$ for $N \geq 1$, the problem reduces to finding a single $C_{-N}$ such that
\begin{equation*}
( D_{x_j} + R_j^{(-N)}) (I + C_{-N})  - (I + C_{-N})D_{x_j} \in \Psi^{-(N+1)},\quad 1\le j\le n,
\end{equation*}
which leads to a system for the  unknown principal symbol of $C_{-N}$:
\begin{equation*}
\begin{cases}
\frac{1}{i} \partial_{x_1} c_{-N} &= - r_{1}^{(-N)} \\ 
\dots \\
\frac{1}{i} \partial_{x_n}c_{-N}  &= - r_{n}^{(-N)}
\end{cases}
\end{equation*}
As described above,
 by (\ref{e:quantum_comm_relations}), the symbol $r_{1}^{(-N)}$ determines all the other $r_{j}^{(-N)}$ in a 
fashion similar to that of $q_{1}^{(0)}$.  Thus, we solve the first equation on the same conic open subset $V$, which in turn solves the 
remaining equations. 
It is important to keep in mind that we are not invoking any abstract existence theorem that might force $V$ to shrink 
after each iteration. Setting 
\begin{equation} \label{e:A_2}
A_2 := A_2^{(0)} \prod_{k > 0} (1 + C_{-k}),
\end{equation} 
we completed the proof of (\ref{e:pseudor_equivalence}). Now set $A :=A_2 A_1$ to obtain our desired operator.  
To complete the proof of  Prop. \ref{p:CdV_qbnf}, we set $U=\chi^{-1}(V)$
and take   $\mathcal{C}$ to be a closed subset of $Gr(\chi)$ over $V$. 
\end{proof}

\begin{rmk}\label{rm: b}
    In Prop. \ref{prop:smooth_spec_measure_microlocal}, we will use a symbol, $b$, defined in terms of the symbol $a_2^{(0)}$ 
    which was just derived.
    We note that, being parametrices of each other, $A$ and $B$ have principal symbols which multiply to equal 1 over the diagonal.
    Thus,  $\sigma(\Psi BA\Psi^*) = \left|\sigma(\Psi)\right|^2$ for any $\Psi\in\Psi^0(M)$ supported on $\Omega$.
\end{rmk}

Recall $\Omega_0 \subset T^*X$ from Def. \ref{defn:QCI}. Given $(x_0,\xi^0) \in \Omega_0$, consider the conic neighborhood $U$ 
provided by Prop. \ref{p:CdV_qbnf}.  We choose $\epsilon_0 = \epsilon_0(U)$ sufficiently small such that the following
two properties hold:

(i) there exists a  chain  of conic neighborhoods of $(x_0;\xi^0)$
(with the containments being compact inclusions when projected onto the cosphere bundle $S^*M$),
\begin{equation} \label{e:nested_flow_nbhds}
W_{\epsilon_0/(n+1)}  \subset W_{2\epsilon_0/(n+1)} \subset \dots \subset W_{\epsilon_0} \subset U,
\end{equation}
with the property that, for any $i_k \in \{1,\dots, n\}$ and $1 \leq k \leq n$,
\begin{align} \label{e:time_intervals}
& \forall (x;\xi) \in W_{\epsilon_0/(n+1)}\hbox{ and  }\overline{t}=(t_1, \dots, t_n) \in \overline{J}(\epsilon_0) 
: = (-\epsilon_0/(n+1), \epsilon_0/(n+1))^n, \\
\nonumber & \exp(t_1 H_{p_{i_1}}) \circ \exp(t_2 H_{p_{i_2}}) \circ \dotsm \circ \exp(t_k H_{p_{i_k}}) (x_0;\xi^0) 
\in W_{k\epsilon_0/(n+1)},
\end{align}  
(In other words, for the ``times" $\overline{t}\in \overline{J}(\epsilon_0)$, 
$W_{\epsilon_0/(n+1)}$ stays properly contained in $U$ after jointly flowing out in $k\leq n$ directions;
an $\epsilon_0$ satisfying this  exists due to the normal form in Prop. \ref{p:DH_nf} and the explicit forms of the $P_j^0$.); 
\smallskip

\noindent and 
\smallskip

(ii) the condition  \eqref{e:ii}, described  at the start of Sec. \ref{sect:microlocal_asymps} below.  
\medskip

Towards the end of the proof of the next lemma we will also need to consider the nested sequence of the images, 
$\left\{ \chi \left( W_{j\epsilon_0/(n+1)} \right)  \right\}_{j=1}^{n+1}$, 
each contained in $V\subset T^*\R^n\setminus 0$,
where $\chi$ is the canonical transformation from Prop. \ref{p:CdV_qbnf}.
\medskip

The following is a slight reformulation of a lemma from \cite{CdV79}, geared towards the computations 
below for the Schwartz kernels of the joint propagators. 

\begin{lem}(\cite{CdV79} Lemma 2.3) \label{l:prop_diff}
For $A, B$ as in Prop. \ref{p:CdV_qbnf} and any 
$Q_{W_{\epsilon_0/(n+1)}}\in\Psi^0(M)$, i.e., smoothing outside of ${W_{\epsilon_0/(n+1)}}$, 
\begin{equation} \label{e:propagator_kernels_equality}
K_{e^{i \overline{t} \cdot \overline{P}}Q_{W_{\epsilon_0/(n+1)}}}(x,y) 
- K_{B e^{i \overline{t} \cdot \overline{P^0}} A Q_{W_{\epsilon_0/(n+1)}}} (x,y)
\in C^{\infty}\left(\,\overline{J}(\epsilon_0) \times M \times M\, \right)
\end{equation}
\end{lem}

\begin{proof} 
The main idea is to read off properties of the propagators $e^{it_k P_k},\, 1\le k\le n,$ 
from the microlocal properties of  the model  $e^{itP_k^0}$ (see Lemma \ref{l:Hormander_half_wave} below) 
and those of the canonical transformations $\chi,\chi^{-1}$ and their associated FIOs 
$A,B$ (from Prop. \ref{p:CdV_qbnf} above).  
Starting with the case $k=1$,  denote the given $Q_{W_{\epsilon_0/(n+1)}}$ by $Q_0$,
and consider a second operator, $Q_{1}\in\Psi^0(M)$,  chosen to be supported in $W_{2 \epsilon_0/(n+1)}$. 
Note that both $Q_0$ and $Q_1$ are supported in the set $U$ as in Prop. \ref{p:CdV_qbnf}. 
By the discussion after  equation (\ref{e:nested_flow_nbhds}), one can choose $Q_1$ so that $Q_{1}\equiv I$ on $\supp Q_0$
and thus $Q_{1} Q_0  \equiv Q_0$. 

Next, for
 $t_1\in J(\epsilon_0):=(-\epsilon_0/(n+1), \epsilon_0/(n+1))$ and $s_1 \in [0, t_1]$ 
 (so that $t_1-s_1\in J(\epsilon_0)\cap\R_+$), 
define an operator acting on $L^2(M)$ by
\begin{align*}
& V_1(s_1,t_1) :=   e^{i s_1 P_1} \, Q_1 \, B \, e^{i(t_1 - s_1) \cdot P^0_1} \, A Q_0.
\end{align*}
As described in \eqref{e: deriv of eitP} below,
for $f \in L^2(M)$ we can  differentiate  $V_1(s_1,t_1)(f)$ in $s_1$ to obtain
\begin{align} \label{eqn:V_1_diffd}
 \frac{1}{i} \partial_{s_1}V_1(s_1,t_1)(f)  &=&  \\
\nonumber & \big(  e^{is_1 P_1} \, P_1 \,  Q_1 \, B \,  e^{i(t_1 - s_1) P_1^0} \, A Q_0  
& \!\!\!\!\!\!-  &\,\,\, e^{i s_1 P_1} \, Q_1 \, B \, P_1^0 \, e^{i(t_1 - s_1) P_1^0} \, A Q_0 \big)(f).
\end{align}
The first term on the RHS of \eqref{eqn:V_1_diffd}  is
$$\equiv \left(e^{is_1P_1}\right)\left(P_1\right)\left(Be^{i(t_1-s_1)P_1^0}AQ_0\right) f \mod C^\infty,$$
since $Q_1\equiv I$ on $\supp \left(Be^{i(t_1-s_1)P_1^0}AQ_0\right)$. 
On the other hand, the negative of the second term on the RHS of \eqref{eqn:V_1_diffd}  is
\beas
&\equiv& \left(e^{is_1P_1}\right) \left(BP_1^0e^{i(t_1-s_1)P_1^0}AQ_0\right) f 
	\equiv \left(e^{is_1P_1}\right)  \left(BP_1^0ABe^{i(t_1-s_1)P_1^0}AQ_0\right) f \\
&\equiv& \left(e^{is_1P_1}\right) \left(P_1\right)\left(Be^{i(t_1-s_1)P_1^0}AQ_0\right) f \mod C^\infty.
\eeas
Here,  the first equivalence is again due to $Q_1\equiv I$ on $\supp \left(Be^{i(t_1-s_1)P_1^0}AQ_0\right)$;
the second is because $AB\equiv I$ on functions with microlocal support in $V=\chi(U)$;
and the third follows from $BP_1^0A\equiv P_1$ on the same class.
Thus, the RHS of \eqref{eqn:V_1_diffd}  is
$$\equiv \left(e^{is_1P_1}\right)\left(P_1-P_1\right)\left(Be^{i(t_1-s_1)P_1^0}AQ_0\right) f \equiv 0\mod C^\infty,$$
and so the expression in \eqref{eqn:V_1_diffd} is smooth.
Integrating the LHS of \eqref{eqn:V_1_diffd} in $s_1$ from $0$ to $t_1$, one obtains
\begin{equation} \label{e:final_diff_prop}
C^{\infty}(M)\owns V_1(t_1,t_1)(f) - V_1(0,t_1)(f) \equiv i \left( e^{i t_1 P_1} Q_0 
- B e^{i t_1 P_1^0} A Q_0 \right)(f)
\end{equation}
for each  $t_1\in J(\epsilon_0)$.  
A similar argument applies if we first differentiate \eqref{eqn:V_1_diffd} with respect to $t_1$ an arbitrary number of times,
gain using \eqref{e: deriv of eitP}.
Hence, 
$$K_{e^{it_1P_1}Q_0}-K_{Be^{it_1P_1^0}AQ_0}\in C^\infty(J(\epsilon_0)\times M \times M).$$

We now repeat the above arguments iteratively, choosing  for each $k \in \{1, \dots, n-1\}$ an operator $Q_{k+1}$ microlocally supported 
on $W_{(k+2)\epsilon_0/(n+1)}$ with $Q_{k+1}\equiv I$ on $W_{(k+1)\epsilon_0/(n+1)}$ and thus
 $Q_{k+1}Q_{k}\equiv Q_{k}$.
 Let $\overline{t}':=(t_1,\dots,t_k)\in J(\epsilon_0)^{k}$ and similarly for $\overline{P}'$ and $\overline{P^0}'$.
  The induction hypothesis at step $k$ is 
  \begin{equation}\label{eqn:Vk}
  e^{i\overline{t}'\cdot \overline{P}'}Q_0-B e^{i\overline{t}'\cdot \overline{P^0}'}AQ_0:L^2(M)\to C^\infty(M).
   \end{equation}
 To prove that the induction step holds,  we form the analogue of $V_1$, namely
$$
 V_{k+1}\left(s_{k+1},t_{k+1}\right):=e^{is_{k+1}P_{k+1}}Q_{k+1}B e^{i(t_{k+1}-s_{k+1})P^0_{k+1}}AQ_0
$$
 and argue as above. The terminal statement, \eqref{eqn:Vk} for $k=n$, yields \eqref{e:propagator_kernels_equality}.
 \end{proof}


\section{Oscillatory Integrals and Propagators}

We now 
record some basic results from the theory of Fourier integral operators. 
We start with an adaptation of \cite[Prop. 25.3.3]{HormanderBook1985b} and the immediate remarks following it in a single 
proposition.  For the specific definitions of terms, the reader is directed to \cite{HormanderBook1985, HormanderBook1985b}.

\begin{prop} \label{p:generating_function}
Let $\mathcal{C}$ be a local homogeneous canonical relation, given by the graph of the canonical transformation $\chi$ described by 
(\ref{eqn:DH_symplecto}), from a conic neighborhood of $(x_0, \xi^0) \in T^*M \backslash 0$ to a conic neighborhood of 
$(y_0, \eta^0) \in T^*\mathbb{R}^n \backslash 0$.  Consider  $K_{B} \in I^0(M \times \mathbb{R}^n, (\mathcal{C}')^{-1})$, 
the Schwartz kernel of an operator $B$.  Assume the family of symbols $(\ref{eq:DH_cara})$ satisfies fiber rank condition 
on some open set $\tO \subset U \cap \Omega_0 \neq \emptyset.$

Then for any local coordinates, $y = X(x;\xi)\in\omega'\subset\R^n$ on a neighborhood of $y_0\in\mathbb{R}^n$
and  $x$ on a neighborhood  $\omega\subset M$ of $x_0$, 
on which (\ref{eqn:DH_symplecto}) is given, there exists a conic neighborhood 
$\Omega' \subset \chi(\tO)$ of 
$(y_0, \eta_0)$ such that for all $Q_{\Omega'}\in \Psi^0(\R^n)$  
supported on $\Omega' $ and for all $y \in \omega', x \in \omega$,
\begin{equation} \label{e:A_oi}
K_{B Q}(x,y) - (2 \pi)^{-n} \int_{\mathbb{R}^n} e^{i ( S(x;\eta) - y \cdot \eta)} b(x,y; \eta) \, d \eta \in C^{\infty}(M \times \mathbb{R}^n )
\end{equation}
with $b \in S^0( M \times \mathbb{R}^n \times \mathbb{R}^n)$.  Furthermore, $S$ is the generating function for $\chi^{-1}$.
\end{prop}

Applying this result in combination with Prop. \ref{p:CdV_qbnf} and \cite[Thm. 25.2.2]{HormanderBook1985b} on the kernels of adjoints, 
we find that the microlocally elliptic Fourier integral operator $A^*$ associated to the inverse of the symplectomorphism from 
Prop. \ref{p:DH_nf} has a kernel that can be expressed as (\ref{e:A_oi}) on $\Omega'$.  
Note that we will later take the conic open set 
$\Omega$, appearing in our Thm. \ref{thm:main}, to be a subset of $\chi(\Omega')$.

\begin{proof}
    Setting $\eta = \Xi(x;\xi)$ and recalling the fiber rank condition, it follows that $\frac{D\eta}{D\xi}$ is non-singular on $\tO$ 
    and we can therefore apply the Inverse Function Theorem.  
    As $(x;\xi)$ were coordinates on $(\mathcal{C}')^{-1}$, we can replace these with $(x;\eta)$.

Now, a combination of \cite[Thm. 21.2.18]{HormanderBook1985}, \cite[Thm. 25.3.7]{HormanderBook1985b}, and the 
discussion on \cite[p. 30]{HormanderBook1985b} shows there exists a homogeneous generating function $S(y;\eta)$ for 
$\chi^{-1}$.  The neighborhood $\Omega' \subset \chi(\tO)$ comes  from choosing the symbol $b$ to be supported on 
$\chi(\tO)$.
\end{proof}

Next, we recall a result that appears in \cite{Hormander1968}, but recast in the language of 
Lagrangian distributions \cite{FIO1, HormanderBook1985b}; see, e.g.,  \cite[\S4.1]{SoggeBook2017}:

\begin{lem} \label{l:Hormander_half_wave}
Let $P \in \Psi^1(M)$ be self-adjoint with homogeneous principal symbol $p$.  The kernel of $e^{itP}$ is an element 
of 
$I^{-1/4}(\mathbb{R} \times {M} \times M, \mathcal{C}')$ where $\mathcal{C}$ is the canonical relation
\begin{equation}\label{eqn: C for P}
    \mathcal{C} = \left\{ (t,x,y; \tau, \xi, \eta) \, : \, \tau + p(x;\xi) =0 \mbox{ and } (x;\xi) = \exp(tH_p)(y;\eta)\right\}.
\end{equation}
\end{lem}

Note that differentiation of this  is allowed  and gives a precise description of the 
kernels of derivatives of $e^{itP}$ with respect to $t$,  within the FIO framework, namely, for $j\ge 1$, 
\begin{equation}\label{e: deriv of eitP}
K_{\frac{d^j}{dt^j}\left(e^{itP}\right)}\in I^{j-\frac14}(\mathbb{R} \times {M} \times M, \mathcal{C}').
\end{equation}

Now, we also recall a useful result that takes the above lemma a step further:

\begin{lem} \label{l:model_prop_Lag} 
Let $\chi$ be the symplectomorphism as in Prop. \ref{p:CdV_qbnf}. Choose $\epsilon_0$ and the conic neighborhood 
$W_{\epsilon_0/(n+1)}$ as described in (\ref{e:nested_flow_nbhds}).  Let $\mathcal{C}_{\overline{P^0}}$ be the
canonical relation associated to the joint Hamiltonian flow 
of the symbols $p^0_1, \dots, p^0_n$; that is,
\begin{equation*}
\mathcal C_{\overline P^0} = \{(\overline t,\overline\tau, x,\xi,y,\eta):\, (x,\xi) = \Phi_1^{t_1}\circ\dotsm\circ\Phi_n^{t_n}(y,\eta),\, \xi 
= \overline \tau\}.
\end{equation*}

For a possibly smaller $\epsilon_0 >0$,  for all $ Q_{\chi(W_{\epsilon_0/(n+1)})} \in \Psi^0(\R^n) $ with 
$\supp Q_{\chi(W_{\epsilon_0/(n+1)})} \subset \chi(W_{\epsilon_0/(n+1)})$ and for all $\overline{t} \in \overline{J}(\epsilon_0)$, we have
\begin{equation*}
K_{ e^{i \overline{t} \cdot \overline{P^0}} Q_{\chi(W_{\epsilon_0/(n+1)})}} \in I^{-n/4}(\mathbb{R}^n \times \mathbb{R}^n \times 
\mathbb{R}^n, \mathcal{C}_{\overline{P}^0}').
\end{equation*}
Moreover, in Cartesian coordinates on $\mathbb{R}^n$, $K_{ e^{i \overline{t} \cdot \overline{P^0}} Q_{\chi(W_{\epsilon_0/(n+1)})})} $  
satisfies:
\begin{equation} \label{e:prop_model_integral}
K_{ e^{i \overline{t} \cdot \overline{P^0}} Q_{\chi(W_{\epsilon_0/(n+1)})}} (\overline{t}, x,z) - \frac{1}{(2\pi)^n} \int_{\mathbb{R}^n}
e^{i \left[ (x-z) \cdot \xi + \overline{t} \cdot \xi \right]} \, q(\overline{t},x,z; \xi) d \xi \in C^{\infty}(\mathbb{R}^n \times \mathbb{R}^n \times 
\mathbb{R}^n)
\end{equation}
where $q \in S^0(\mathbb{R}_{\overline{t}}^n \times \mathbb{R}_x^n \times \mathbb{R}_z^n \times \mathbb{R}_{\xi}^n \backslash \{0\})$ 
has a polyhomogeneous expansion with leading term
\begin{align} \label{e:princ_symbol_model}
\begin{split}
 & q_n^{(0)}(t_n, x, x + t_{n}e_{n};\xi) 
 \times  q_{n-1}^{(0)}(t_{n-1}, x+ t_{n}e_{n}, x + t_{n}e_{n} + t_{n-1} e_{n-1}; \xi) \times \\
& \quad  \dots \times q_2^{(0)}(t_2, x+ \sum_{i=3}^{n} t_i e_i, x +  \sum_{i=2}^{n} t_i e_i, \xi) 
\times q_1^{(0)}(t_1, x+ \sum_{i=2}^{n} t_i e_i, z; \xi).
\end{split}
\end{align}
Here, $q_j^{(0)}$ is the principal symbol of the kernel of (microlocalized) $e^{i t_j P_j^0}$ and $e_i$ is the $i$-th standard basis vector.

\end{lem}

\begin{proof} 

One observes that the classic construction of the parametrix for the half-wave operator, as a Lagrangian distribution, by H\"ormander 
\cite{Hormander1968} (see also \cite{SoggeBook2017}) is valid for our self-adjoint $P^0_j$ despite them not being elliptic.  While this 
effects their spectrum as operators on $L^2(\mathbb{R}^n)$, we are only concerned with microlocal quantities.  

In Cartesian coordinates 
 $x \in \mathbb{R}^n$, one can may use $(x-y) \cdot \xi + t_j \xi_j$ as the phase function of $e^{it_j P^0_j}$, 
which parametrizes $\mathcal{C}_{P^0_j}$, the canonical relation of $e^{it_j P_j^0}$. Furthermore, since $P_j^0$ is a differential 
operator with zero subprincipal symbol, there are substantial simplifications in the transport equations which lead to the construction of 
the propagator's full symbol: 
\begin{equation} \label{eqn:jth_propag_full_symb}
\sigma_{full} \left( e^{i t_j P_j^0} Q_{\chi(W_{j\epsilon_0/(n+1)})} \right) = \sum_{k=0}^{\infty} q_j^{(-k)} 
\end{equation}
where $q_j^{(-k)} \in S^{-k}(\mathbb{R}_t \times \mathbb{R}_x^n \times \mathbb{R}_y^n \times \mathbb{R}^n_{\xi})$.  In particular, we 
have such an expansion for $|t_j| \leq \epsilon_0/(n+1)$ where $\epsilon_0$ is now possibly smaller.  
Note again that the Schwartz kernel 
of the difference
\begin{equation} \label{eqn:cutoff_propag_equiv}
e^{i t_j P_j^0} Q_{\chi(W_{j\epsilon_0/(n+1)})} - Q_{\chi(W_{(j+1)\ve_0/(n+1)})} e^{i t_j P_j^0} 
Q_{\chi(W_{j\epsilon_0/(n+1)})} : L^2(\mathbb{R}^n) \rightarrow C^{\infty}(\mathbb{R}^n)
\end{equation}
is a smooth function.  Here the  symbols of $Q_{\chi(W_{j\epsilon_0/(n+1)})}$ and $Q_{\chi(W_{(j+1)\epsilon_0/(n+1)})}$ are again 
configured so that $Q_{W_{2\epsilon_0/(n+1)}} Q_{W_{\epsilon_0/(n+1)}}  = Q_{W_{\epsilon_0/(n+1)}}$ modulo $C^\infty$. 
The smoothness claim follows from the propagator's kernel being a Lagrangian distribution, with
$WF' (e^{i t_j P_j^0}) \subset \mathcal{C}_{P_j^0}$, and 
\[ \left (I - Q_{\chi(W_{(j+1)\epsilon_0/(n+1)})} \right) Q_{\chi(W_{j\epsilon_0/(n+1)})} \in \Psi^{-\infty}(\R^n). \]  
These facts allow us to compute the principal symbol (in fact, the full symbol in our 
coordinate system) of the operator $ e^{i \overline{t} \cdot \overline{P^0}} Q_{\chi(W_{\epsilon_0/(n+1)})}.$ 

To calculate the kernel of the composition  $e^{i[t_1P_1 + t_2 P_2]}Q_{\chi(W_{\epsilon_0/(n+1)})}$,
we follow the proof of \cite[Thm. 25.2.2]{HormanderBook1985b}.  
Let $K_2(t_2, x^2,y^2)$ and $K_1(t_1, y^2,y^1)$ be the Schwartz kernels of $e^{i t_1 P_1^0} Q_{\chi(W_{\epsilon_0/(n+1)})}$ 
and $e^{i t_2 P_2^0} Q_{\chi(W_{2\epsilon_0/(n+1)})}$, respectively.  Furthermore, let $q_2(t_1, x^2,y^2;\xi^2)$  
and $q_1(t_1, y^2, y^1; \xi^1) $ denote the respective full symbols as in (\ref{eqn:jth_propag_full_symb}).

The function $(x^2-y^2)\cdot\xi^2 + t_2\xi^2_2 + (y^2-y^1)\cdot \xi^1 + t_1\xi^1_1)$ is a clean phase function, as is shown by a 
straightforward calculation. 
We now need to localize to where $|\xi^2| \approx |\xi^1|$.  That is, we find that there exists a
homogeneous of degree 0 cutoff $\chi_0^{2,1}(t_2, t_1, x^2,y^2,y^1; \xi^2, \xi^1)$ which is equal to 1 on the region 
$\{ |\xi^2| /(2C_{2,1})< |\xi^1| < \left(C_{2,1}/2\right) |\xi^2|\}$ and supported on $\{ |\xi^2|/C_{2,1} < |\xi^1| < C_{2,1} |\xi^2|  \}$ 
for  some $C_{2,1} > 0$ depending on $\mathcal{C}_{P_1^0}$ and $\mathcal{C}_{P_2^0}$ such that,
for  $\overline{t} \in I(\epsilon_0)$, 
\begin{align*}
 \int K_2(x^2,y^2) \, K_1(y^2, y^1) \, dy^2 & \\
 - \int \int \int & \,e^{i[(x^2-y^2)\cdot \xi^2  + t_1 \xi_1^1 + t_2 \xi_2^2 + (y^2 - y^1) \cdot \xi^1]} 
 \chi_0^{2,1}(t_2, t_1, x^2,y^2,y^1; \xi^2, \xi^1) \, \times \\
& \quad q_2(t_2, x^2,y^2;\xi^2) \, q_1(t_1, y^2, y^1; \xi^1) \, d\xi^1 \, d\xi^2 \, dy^2 
\end{align*}
belongs to $ C^{\infty}(\mathbb{R}^n \times \mathbb{R}^n)$.  
This can be seen explicitly via integration by parts in  $y^2$.

Applying stationary phase in the variables $(y^2, \xi^2)$ to the triple integral above yields the phase function 
$(x^2-y^1)\cdot \xi^1 + t_1 \xi_1^1 + t_2 \xi_2^1$ and a full symbol whose expansion is given by
\begin{equation*}
\sum_{\nu \geq 0} \frac{1}{(2i)^{\nu}} \left( \langle D_{y^2}, JD_{\xi^2} \rangle^{\nu} \chi_0^{2,1}(x^2,y^2,y^1; 
\xi^2, \xi^1) q_2(t_1, x^2,y^2;\xi^2) q_1(t_1, y^2, y^1; \xi^1)  \right)_{\upharpoonright y^2 =x^2 + t_2 e_2, \, \xi^2 = \xi^1}.
\end{equation*}
 Since $\chi_0^{2,1}$ equals 1 at the critical point, one sees that the principal term is of the form 
 $q_2^{(0)}(t_2, x^2, x^2+t_2 e_2; \xi^1) q_1(t_1, x^2 + t_2 e_2, y^1; \xi^1)$.  We leave it to the reader to verify that 
 $$(x^2-y^2)\cdot\xi^2 + t_2\xi^2_2 + (y^2-y^1)\cdot \xi^1 + t_1\xi^1_1$$ 
 parametrizes $\mathcal{C}_{P_2^0} \circ \mathcal{C}_{P_1^0}$ and satisfies the appropriate clean condition.
 
 Repeating this calculation $n-2$ more times whilst invoking (\ref{eqn:cutoff_propag_equiv}), 
 we obtain the expression in (\ref{e:princ_symbol_model}) and the expansion for the remaining 
 terms in $q$ from \eqref{e:prop_model_integral} are
\begin{align*} 
R(\overline{t}, x^n,y^1; \xi^1) := \sum_{\nu \geq 1} \frac{1}{(2i)^{\nu}} & \Bigg[\langle D_{y^n}, JD_{\xi^n} \rangle^{\nu} 
\chi(W_{n\epsilon_0/(n+1)})(x; \xi^n) \chi_0^{n,n-1}(x^n,y^n,y^1;\xi^n, \xi^1) q_n(t_n, x^n,y^n;\xi^n)\\
& \times q_{12\dots(n-1)}(t_1,\dots, t_{n-1}, y^n,y^1;\xi^1) \Bigg]_{y = z + t_{n} e_{n}, \xi^n = \xi^1}.
\end{align*}
One sees that  $R(\overline{t}, x^n,y^1; \xi^1) \in S^{-1}(\mathbb{R}^n \times 
\mathbb{R}^n \times \mathbb{R}^n \times \mathbb{R}^n \backslash \{0\})$.
 Here, $q_{12\dots(n-1)}$ is the full symbol of the microlocalized propagator for the 
 first $n-1$ Hamiltonians $P_1, \dots, P_{n-1}$, 
 namely $e^{i \sum_{j=1}^{n-1} {t_jP^0_j}} Q_{\chi(W_{\epsilon_0/(n+1)})}$.
\end{proof}

In the next lemma, we record the asymptotics for the microlocalized spectral measure for the model propagators:

\begin{lem} [(Microlocalized spectral measure on $\mathbb{R}^n$)] \label{l:microlocalized_measure_Euclidean}

Let $\overline{\mu} \in \R^n$ and
$A_j\in\Psi^{\alpha_j}(\R^n)$  for some $\alpha_j\in\R,\,  j=1,2$, with principal symbols
 $a^{(0)}_j(x,\xi)$. Let $\epsilon > 0$, and $\rho\in\mathscr S(\R)$  such that 
 $\supp \wh\rho\subset (-\epsilon,\epsilon)$ and $\wh\rho(s) = 1$ for all $|s|\le \epsilon/2$. Let $K(x,y;\overline{\mu})$ 
 be the Schwartz kernel of $A_1\circ\rho(P_1^0 - \mu_1)\circ\dotsm\circ\rho(P_n^0 - \lambda \mu_n)\circ A_2^*$, where the operator
\begin{equation*}
    \rho(P_1^0 - \mu_1)\circ\dotsm\circ\rho(P_n^0 - \mu_n) := \frac{1}{(2\pi)^{n}} \int\limits_{\R^{n}} e^{i \overline{t} \cdot 
    \overline{P^0}} e^{-i \overline{t} \cdot \overline{\mu}} \, \wh\rho(t_1)\dotsm\wh\rho(t_n)\,d\tbar
\end{equation*}
is defined via the spectral calculus.

Then, there exists
 $C_0,\, C_1 >0 $ such that for all $|\overline{\mu}| \ge C_0$, we have 
\[ K(x,y, \overline{\mu}) = \frac{|\overline{\mu}|^{\alpha_1+\alpha_2}}{(2\pi)^n}e^{i\lambda ( x -y ) \cdot \overline{\mu}}a_1 
\left( x, \frac{\overline{\mu}}{|\overline{\mu}|}) \right) \overline{a_2 \left( y, \frac{\overline{\mu}}{|\overline{\mu}|} \right)} 
+  R(x,y, \overline{\mu}),\]
where 
\[\sup\limits_{|x-y| \le \ve/2}\left| R(x,y,\overline{\mu})\right| \le C_1|\mu|^{\alpha_1+\alpha_2-1}.\]

\end{lem}

\begin{rmk}
\textnormal{
Note that in the applications below, we shall choose $A_1,A_2$ to be order zero pseudodifferential cutoffs, 
but it does not complicate the proof to give the statement for operators of non-zero order. 
}
\end{rmk}

\begin{proof}
The proof is standard but we essentially reproduce the proof \cite[Prop. 25.1.5]{HormanderBook1985b} (``symbol maps for 
Lagrangian distributions") for our explicitly given FIOs. Observe that in the sense of distributions, we have 
\begin{align*}
\int\limits_{\R^{2n}} A_1(x,z)e^{i\langle z-w,\xi\rangle}A_2^*(w,y)\,dw\,dz & = \frac{1}{(2\pi)^{2n}}\int\limits_{\R^{4n}} e^{i (x-z) \cdot \eta} 
a_1(x,\eta)e^{i\langle z-w,\xi\rangle}\overline{a_2(y,\zeta)}e^{i (w-y) \cdot \zeta}\,dw\,d\zeta\,dz\,d\eta\\
& = \frac{1}{(2\pi)^{2n}}\int\limits_{\R^{2n}} e^{i(z-w) \cdot \xi }\wh a_1(x,z-x) \overline{\wh a_2(y,w-y)}\,dw\,dz\\
& = \frac{1}{(2\pi)^{2n}}e^{i (x-y) \cdot \xi }\int\limits_{\R^{2n}} e^{i\langle z -w ,\xi\rangle} \wh a_1(x,z)\overline{\wh a_2(y,w)}\,dw\,dz\\
& = e^{i (x-y) \cdot \xi }a_1(x,\xi)\overline{a_2(y,\xi)}.
\end{align*}
Thus, we have that 
\begin{align*}
&A_1\circ\rho(D_{x_1} - \mu_1)\circ\dotsm\circ\rho(D_{x_n} - \mu_n)\circ A_2^*(x,y) \\
& \hskip 0.5in = \frac{1}{(2\pi)^{2n}}\int\limits_{\R^{2n}}e^{i (x - y) \cdot \xi  + i \tbar \cdot (\xi-\overline{\mu})}a_1(x,\xi)\overline{a_2(y,\xi)}
\wh\rho(t_1)\dotsm\wh\rho(t_n)\,d\xi\,d\tbar.
\end{align*}
Changing variables via $\xi\mapsto |\mu| \xi$, the last integral becomes
\begin{equation}\label{eq:a1a2_int}
\frac{|\mu|^{n}}{(2\pi)^{2n}}\int\limits_{\R^{2n}}e^{i|\overline{\mu}|\lp ( x- y) \cdot \xi  + \tbar \cdot (\xi-\frac{\overline{\mu}}{|\overline{\mu}|})
\rp} a_1(x,|\overline{\mu}|\xi)\overline{a_2(y,|\overline{\mu}|\xi)}\wh\rho(t_1)\dotsm\wh\rho(t_n)\,d\xi\,d\tbar.
\end{equation}
By definition, we have that 
\[a_j(x,\xi) - a_j^{(0)}(x,\xi) \in S^{\alpha_j-1}(T^*M),\quad j=1,2,\] 
and since $a_j$ is homogeneous of degree $\alpha_j$, we have that \eqref{eq:a1a2_int} is equal to 
\begin{equation}\label{e:model_int}
\frac{|\overline{\mu}|^{n+\alpha_1+\alpha_2}}{(2\pi)^{2n}}\int\limits_{\R^{2n}} e^{i|\overline{\mu}|\lp ( x- y) \cdot \xi  + \tbar \cdot (\xi-
\frac{\overline{\mu}}{|\overline{\mu}|})\rp}  a_1^{(0)}(x,\xi)\overline{a_2^{(0)}(y,\xi)}\wh\rho(t_1)\dotsm\wh\rho(t_n)\,d\xi\,d\tbar + R_1(x,y,
\overline{\mu}),
\end{equation}
where $R_1$ is defined by 
\begin{equation}\label{e:R_defn}
R_1(x,y,\overline{\mu}) = \frac{|\overline{\mu}|^n}{(2\pi)^n}\int\limits_{\R^{2n}}e^{i|\overline{\mu}|\lp ( x- y) \cdot \xi  + \tbar \cdot (\xi-
\frac{\overline{\mu}}{|\overline{\mu}|})\rp} r_1(x,y,|\overline{\mu}|\xi)\wh\rho(t_1)\dotsm\wh\rho(t_n)\,d\xi\,d\tbar,
\end{equation}
with $r_1(x,y,\xi): = a_1(x,\xi)a_2(y,\xi) - a_1^{(0)}(x,\xi)\overline{a_2^{(0)}(y,\xi)}\in S^{\alpha_1+\alpha_2-1}$. 

We ignore $R_1$ for 
now and focus on the leading term in \eqref{e:model_int}. 
Let us denote the phase function by 
\[\phi(\tbar,\xi;x,y) :=  x-y \cdot \xi + \tbar \cdot \left( \xi-\frac{\overline{\mu}}{|\overline{\mu}|} \right).\]
Observe that 
\[\nabla_\xi\phi = x-y + \overline{t},\quad \text{and}\quad \nabla_{\tbar}\phi = \xi-\frac{\overline{\mu}}{|\overline{\mu}|}.\]
Thus, for each fixed $x,y$, there is a unique critical point at $\xi_0 = \frac{\overline{\mu}}{|\overline{\mu}|},\, \tbar_0 = y-x.$ Computing the 
Hessian of $\phi$ at this critical point, we have 
\[\Hess_{\xi,\tbar}\phi(\xi_0,\tbar_0;x,y) = \lp\begin{array}{cc}
0 & I_n\\
I_n & 0
\end{array}\rp,\]
where $I_n$ denotes the $n\times n$ identity matrix. This matrix has determinant $(-1)^n$ and signature 0, since its eigenvalues are 
$\pm 1$, each with multiplicity $n$.  Furthermore, this continues to hold for each $\overline{\mu} \neq 0$: 
that is, for every $\overline{\mu}$, the $(\xi_0, \overline{t}_0)$ is unique.  Since $\xi \neq 0$, we can introduce a cutoff 
$\beta \left( \frac{\xi}{|\overline{\mu}|} \right)$ into (\ref{e:model_int}) that equals one for $|\xi| \in [C^{-1}, C]$ for some sufficiently large 
$C>0$ so that $\beta \left( \frac{\xi}{|\overline{\mu}|} \right) = 1$ at $(\xi_0, \overline{t}_0)$.  Note that the difference between 
(\ref{e:model_int}) as written and (\ref{e:model_int}) with the cutoff $\beta$ is $\mathcal{O}_{N}(|\overline{\mu}|^{-N})$ 
for $|\overline{\mu}|  \geq C_0$, for an appropriate $C_0$, as
 seen by integrating by parts in $\overline{t}$. 
Therefore, by stationary phase, we have that the integral in \eqref{e:model_int} is equal to 
\[\frac{|\overline{\mu}|^{n+\alpha_1+\alpha_2}}{(2\pi)^{2n}}|\overline{\mu}|^{-n}e^{i|\overline{\mu}|\langle x -y,\cbar\rangle}a_1^{(0)}(x,
\overline{\mu})\overline{a_2^{(0)}(y,\overline{\mu})}\wh\rho(y_1-x_1)\dotsm\wh\rho(y_n-x_n) 
+ \mathcal O(|\overline{\mu}|^{\alpha_1+\alpha_2-1}).
\]
If $\epsilon$ is chosen sufficiently small, then $|x-y|< \epsilon/2$ implies that $\wh\rho(y_k-x_k) = 1$ for all $k$. 
Thus, we can simplify the 
above expression to 
\begin{align*}
A_1&\circ\rho(D_{x_1} - \mu_1)\circ\dotsm  \circ\rho(D_{x_n} - \mu_n)\circ A_2^*(x,y)\\
&
= \frac{|\overline{\mu}|^{\alpha_1+\alpha_2}}{(2\pi)^n}e^{i|\overline{\mu}|( x -y) 
\cdot \frac{\overline{\mu}}{|\overline{\mu}|}}\sigma(a_1)(x,\overline{\mu})\overline{\sigma(a_2)(y,\overline{\mu})}\\
& \hskip 1.4in + R_2(x,y,\overline{\mu}),
\end{align*}
where 
\[\sup\limits_{d_g(x,y)\le \frac{\epsilon}{2}}\left| R_2(x,y,\overline{\mu})\right| = \mathcal O(|\overline{\mu}|^{\alpha_1+\alpha_2-1}).\]
To complete the proof, we consider the contribution of $R_1(x,y,\overline{\mu})$ to \eqref{e:model_int}. 
Noting that the symbol $r_1(x,y,\xi)$ in \eqref{e:R_defn} is of order $\alpha_1+\alpha_2-1$ in $\xi$, 
we observe that repeating the preceding 
stationary phase argument one more time to \eqref{e:R_defn} yields a similar asymptotic wherein each term is 
one order lower in $|\overline{\mu}|$, thus yielding a  contribution of order $\alpha_1+\alpha_2-1$. 

\end{proof}
\color{black}


\section{Microlocalized smoothed spectral measure asymptotics} \label{sect:microlocal_asymps} 

Recall the chain of strict containments,
\begin{equation}
    W_{\epsilon_0/(n+1)}  \subset W_{2\epsilon_0/(n+1)} \subset \dots \subset W_{\epsilon_0} .
\end{equation}
described below Remark \ref{rm: b},
as well as the set $\Omega'$ from Prop. \ref{p:generating_function}.  
At this stage, we are  in a position  to describe precisely and  impose the restriction (ii) we impose on $\epsilon_0$, 
 referred to below \eqref{e:time_intervals},
 in addition to the condition (i) stated there.
 Namely, we take  $\epsilon_0=\epsilon_0(U)>0$ small enough so that 
\begin{equation}\label{e:ii}
\qquad W_{\epsilon_0} \subset \chi({\Omega'})\subset \tO.
\end{equation}
We then define the set $\Omega\subset T^*M\setminus 0$ on which we will henceforth work:
\begin{equation} \label{eqn:Omega_set}
    \Omega : = W_{\epsilon_0/(n+1)} \subset \chi^{-1}(\Omega').
\end{equation} 
In particular, Lemma \ref{l:prop_diff} applies on $\Omega$.

Take $(w_0, \zeta_0) \in \Omega$; then the conic open neighborhoods $\Omega$ and $\Omega'$ of $(w_0, \zeta_0)$ and 
$\chi(w_0, \zeta_0)$, resp., have the following properties: the normal form from Prop. \ref{p:CdV_qbnf} holds, our joint 
propagator equivalence Lemma \ref{l:prop_diff} holds, and the FIOs $A^*,B$ from Prop. \ref{p:CdV_qbnf} can be simultaneously 
microlocalized to the same regions and thus have the representations described in Prop. \ref{p:generating_function} (with the 
same phase functions).  We  now  obtain asymptotics for the microlocalized spectral 
measure of the original QCI system satisfying the fiber-rank condition (see Def. \ref{defn:QCI}):

\begin{prop}\label{prop:smooth_spec_measure_microlocal}
Let $Q_{\Omega} \in \Psi^0(M)$  with  $\supp Q_{\Omega} \subset \Omega$
and $\rho \in \mathscr{S}(\mathbb{R})$ with $\supp \hat{\rho}  \subset (-\epsilon_0, \epsilon_0)$. 
Then, there exist $a \in S^0((\mathbb{R}^n \times M) \times \mathbb{R}^n),b \in S^0(( M \times \mathbb{R}^n) \times \mathbb{R}^n), 
q^{(0)} \in S^0\left(\left(\mathbb{R}^n \times \mathbb{R}^n \times \mathbb{R}^n\right) \times \mathbb{R}^n\right)$, a generating function 
$S(y;\eta)$
defined on $\Omega$, and $C_2:=C_2(a,b,q^{(0)},\Omega)>0$ such that the following holds:

 For all $\overline{\mu} \in \overline{p}(\Omega) \cap B(0, C_2)^{\complement}$ and all $y,w \in proj_M(\Omega)$, the asymptotic 
\begin{align} \label{eqn:smooth_measure_asymps}
  \big(  \rho(P_1 - \mu_1 ) \circ \dots & \circ  \rho(P_n - \mu_n ) \circ Q_{\Omega}\big) (y,w)  = \\
   \nonumber & \big[ e^{i[|\overline{\mu}| (\nabla_{\xi}S(y;\overline{\mu})- \nabla_{\eta} S(w; \overline{\mu})) 
   \cdot \frac{\overline{\mu}}{|\overline{\mu}|}]} \\
 \nonumber & \times b(y, \nabla_{\xi}S(y; \overline{\mu});  \overline{\mu}) \times a(\nabla_{\eta}S(w; \overline{\mu}), w;  \overline{\mu})  \\
 \nonumber & \times q^{(0)} \left( \nabla_{\xi} S(y;\overline{\mu})-\nabla_{\eta} S(w; \overline{\mu}),\nabla_{\xi} S(y;\overline{\mu}),
 \nabla_{\eta} S(w; \overline{\mu}); \overline{\mu} \right) \\
 \nonumber & \times \rho((\nabla_{\xi} S)_1(y;\overline{\mu})-(\nabla_{\eta} S)_1(w; \overline{\mu})) \times \dots \times 
 \rho((\nabla_{\xi} S)_n(y;\overline{\mu})- (\nabla_{\eta} S)_n(w; \overline{\mu})) \big] \\
 \nonumber &\qquad  + \mathcal{R}(y,w;  \overline{\mu}).
\end{align}
holds, with $\mathcal{R}(y,w;  \overline{\mu}) = \mathcal{O}(|\overline{\mu}|^{-1})$ for all $|\overline{\mu}|\ge C_2$.
In particular, the expression in \eqref{eqn:smooth_measure_asymps} is $\mathcal O(1)$ uniformly in $y,w\in proj_M(\Omega)$.
\end{prop}

The following corollary will be important below:
\begin{cor} \label{cor:smooth_measure_neg}
    Let $\Omega \subset T^*M$ be as  in \eqref{eqn:Omega_set} and 
    Prop. \ref{prop:smooth_spec_measure_microlocal}.  Given $N>0$, there exist $C_3,\, C_4>0$ such that 
    for all $\overline{\mu} \in \left(\overline{p}\left(\Omega\right)\right)^{\complement} \cap (B(0, C_3))^{\complement}$, 
    $\forall y,w \in proj_M(\Omega)$, 
    and 
     any $Q_{\Omega} \in \Psi^0$ with $\supp Q_{\Omega} \subset \Omega$,
    \begin{align*}
       \Big|  \big(Q_{\Omega} \circ \rho(P_1 - \mu_1 ) \circ \dots & \circ  \rho(P_n - \mu_n ) \circ Q_{\Omega}^*\big)(y,w)  
       \Big| \leq C_4 \langle \overline{\mu} \rangle^{-N}.
    \end{align*}
\end{cor}

\begin{proof}[Proof of Cor. \ref{cor:smooth_measure_neg}]
Take 
the expression in \eqref{eqn:smooth_measure_asymps} (harmlessly replacing the $Q_\Omega$ on the right by $Q_\Omega^*$), 
and then  compose it with  $Q_\Omega$ on the left.
  It suffices to show that  the Fourier transform in the right variable (which we will label $x$) is 
  $ \mathcal{O}(|\overline{\mu}|^{-N})$ for every $N>0$.
This is done by a standard integration by parts argument in $x$, 
noting that the gradient of the phase is $\nabla_x \Phi_B(y,x;\xi) - \frac{\overline{\mu}}{|\overline{\mu}|}$,
where $B$ is the FIO with canonical relation  $\mathcal{C}_B$  which is the graph of 
the canonical transformation in Prop. \ref{p:CdV_qbnf}, 
thus connecting $W$ with the image of the momentum map $\overline{p}$.  
Given the positive distance away from $\supp \Psi^* \subset W$ (since we are inside of $\overline{p}^{-1}
    (\overline{p}(W)^{\complement}) \subset W^{\complement}$ ), 
    it follows that $\nabla_x \Phi_B(y,x;\xi) - \frac{\overline{\mu}}{|\overline{\mu}|} \neq 0$ 
    over $\supp \Psi^*$.  Now  integrate by parts in $x$.
\end{proof}

\begin{proof}[Proof of Prop. \ref{prop:smooth_spec_measure_microlocal}]
Recall that we have
\begin{equation}
 \left( e^{i \overline{t} \cdot \overline{P}} - B e^{i \overline{t} \cdot \overline{P^0}} A \right) Q_{\Omega}  
 \in C^{\infty}(\mathbb{R}^n \times M \times M)
\end{equation} 
for $\overline{t} \in \overline{J}(\epsilon_0)$, thanks to Lemma \ref{l:prop_diff}.  Take $\rho \in \mathscr{S}(\mathbb{R})$ as in 
Lemma \ref{l:microlocalized_measure_Euclidean} with $\ve = \epsilon_0/(n+1)$.   
Then
\begin{align}
\nonumber & \int \, \hat{\rho}(t_1) \dots \hat{\rho}(t_n) \, e^{i \overline{t} \cdot (\overline{P} - \overline \mu)} d \overline{t} 
\circ Q_{W_{\epsilon_0/(n+1)}} \\
& - \, \, B \circ \rho(P_1^0 - \mu_1) \circ \dots \circ  \rho(P_n^0 - \mu_n) \circ A \circ Q_{W_{\epsilon_0/(n+1)}}
\end{align} 
is smoothing on $L^2(M)$ and has an $L^2 \rightarrow L^{\infty}$ operator norm which is $\mathcal{O}(|\mu|^{-N})$ 
for $|\overline{\mu}| \geq \mu^0(N, \Omega)$, thanks to its Schwartz kernel satisfying the same property.

We remind ourselves 
from Lemma \ref{l:microlocalized_measure_Euclidean} that for $|\overline{\mu}| \geq C_0$, 
we also have the following asymptotics uniformly for all $x,z \in proj_{\R^n} \, (\chi(\Omega))$:
\begin{align*}
&  \rho(P_1^0 - \mu_1) \circ \dots \circ  \rho(P_n^0 - \mu_n) \circ Q_{\chi(\Omega)}  \\
& =  e^{i[ (x-z) \cdot \overline{\mu}]} Q^{(0)}(x-z,x,z; \overline{\mu}) \hat{\rho}(x_1-z_1) \dots \hat{\rho}(x_n-z_n)  + R(x,z;  
\overline{\mu}) \\
& =: K_0(x,z;  \overline{\mu}) + R(x,z;  \overline{\mu})
\end{align*}
where 
\begin{equation*}
R(x,z; \overline{\mu}) = \mathcal{O}(|\overline{\mu}|^{-1}).
\end{equation*}
Thus, it remains to compute the asymptotics of
\begin{equation} \label{e:conjugated_micro_projector_term}
B \circ \rho(P_1^0 - \mu_1) \circ \dots \circ  \rho(P_n^0 - \mu_n) \circ A \circ Q_{\Omega}.
\end{equation}
Before proceeding further, we let $a = \sigma(A \circ Q_{\Omega})$ and $b = \sigma(Q_\Omega \circ B)$.

If we write the Schwartz kernel of the leading term for (\ref{e:conjugated_micro_projector_term}) as
\begin{align}
& \int \int K_{B}(y,x) \, K_0(x,z;  \overline{\mu}) \, K_{ A\, \circ\, Q_{\Omega}}(z,w) \, dz \, dx \\
& = \int   \int \, K_{ B}(y,x) \, \left[ \overline{K_{ (A\, \circ\, Q_{\Omega}^*}}(w,z) \, K_0(x,z;  \overline{\mu}) \, dz  \right] \, dx 
 \, \, (\text{mod } |\overline{\mu}|^{-\infty}).
\end{align}
We then invoke Prop. \ref{p:generating_function} to obtain phase functions $S(y;\xi) - x\cdot \xi$ and $S(w;\eta) - z \cdot \eta$ for 
$B$ and $A^*$, respectively. Finally, after scaling by the large parameter $|\overline{\mu}|$ we make a final application of stationary 
phase in the variables $(x;\xi)$ and $(z,\eta)$. 
From this it follows that the leading term of (\ref{e:conjugated_micro_projector_term}) equals
\begin{align*}
& \big[ e^{i[ (\nabla_{\xi}S(y;\overline{\mu})- \nabla_{\eta} S(w; \overline{\mu})) \cdot \overline{\mu}]} \\
& \times b(y, \nabla_{\xi}S(y; \overline{\mu});  \overline{\mu}) \times a(\nabla_{\eta}S(w; \overline{\mu}), w;  \overline{\mu})  \\
& \times Q^{(0)} \left( \nabla_{\xi} S(y;\overline{\mu})-\nabla_{\eta} S(w; \overline{\mu}),\nabla_{\xi} S(y;\overline{\mu}),\nabla_{\eta} S(w; \overline{\mu}); \overline{\mu} \right) \\
& \times \rho((\nabla_{\xi} S)_1(y;\overline{\mu})-(\nabla_{\eta} S)_1(w; \overline{\mu})) \times \dots \times \rho((\nabla_{\xi} S)_n(y;\overline{\mu})-(\nabla_{\eta} S)_n(w; \overline{\mu})) \big] \\
& + \mathfrak{R}(y,w;  \overline{\mu}).
\end{align*}
where $\mathfrak{R}(y,w;  \overline{\mu}) = \mathcal{O}(|\overline{\mu}|^{-1})$, 
for all sufficiently large $|\overline{\mu}|$, which we write as 
 $|\overline{\mu}| \geq C_2:=C_2(a,b,q^{(0)}, \Omega)$ for some $C_2$; 
 without loss of generality  we can  assume  $C_2\ge C_0(1, \Omega)$.

\end{proof}


\section{Tauberian Arguments and  proofs of Theorem \ref{thm:main} and Cor. \ref{cor:main}} \label{sect:tauberian} 

\subsection{On-diagonal Cluster Bounds} 
 For $\mubar \in \Gamma$ from (\ref{e:image_mm}), define the cube
 \begin{equation} \label{eqn:R_m}
 R_\mubar := \{\overline \tau\in\R^n:\, \tau_k\in [\mu_k,\mu_k+1]\, \text{ for each }k=1,2,\dotsc,n\}.
 \end{equation}
 
 In this section, we aim to show for certain $\Psi\in \Psi^0(M)$, on-diagonal upper bounds for the microlocalized unit-box projector kernel
\begin{equation}\label{spectral_box}
\Psi\Pi_{R_{\mubar}}\Psi^*(x,x)= \sum\limits_{\{j:\,\overline{\lambda_j}  \in R_{\mubar}\}}|\Psi\varphi_j(x)|^2,
\end{equation} 
follow from similar bounds on the microlocalized smoothed spectral measure restricted to the diagonal:
\[\Psi\beta(\overline P - \mubar)\Psi^*(x,x) = \sum\limits_{j=0}^\infty \beta(\overline\lambda_j -\mubar)|\Psi\varphi_j(x)|^2,\]
for an appropriate $\beta\in\mathscr S(\R^n)$. The exposition in this section proceeds in analogy to \cite[\S 3.2]{SoggeBook2014}, which 
gives a robust Tauberian argument in the single-operator case that can be used to justify similar statements from \cite{Hormander1968}. 

We choose a cutoff function $\beta\in \mathscr S(\R^n)$ such that 
\[\beta \ge 0, \quad \beta(0) = 1, \quad \wh\beta(\overline t) \equiv 0 \text{ for }|\overline t| \ge \delta,\]
for some $\delta > 0$ to be specified later. We claim that upper bounds for $\Psi\beta(\overline P - \mubar)\Psi^*$ restricted to the 
diagonal imply similar estimates for $\Psi\Pi_{R_\mubar}\Psi^*(x,x).$
\begin{lem} \label{l:smoothed_measure_rough_measure}
Suppose that given some $\Psi\in \Psi^0(M)$, there exists an open set $\tilde\omega\subseteq M$,
a cone $\Gamma' \subset \Gamma$, a compact set $K\subset \R^n$, and constants $C > 0$ 
and $\overline m\in (\R_{\geq 0})^n$ such that
\[\Psi\beta(\overline P - \mubar)\Psi^*(x,x) \le C(1+\mubar)^{\overline m} := \Pi_{k=1}^n (1+ \mu_k)^{m_k}\]
for all $x\in\tilde\omega$ and all $\,\mubar \in\Gamma' \setminus K$. Then, there exists a constant $C'>0$ so that 
\[\Psi\Pi_{R_\mubar}\Psi^*(x,x) \le C'(1+\mubar)^{\overline m}\]
for all $x\in \tilde\omega$ and all $\,\mubar\in\Gamma' \setminus K.$
\end{lem}

\begin{proof}
Since $\beta(0)=1$, there exists an $\ve_0 > 0$ such that $\beta(\overline \tau) \ge \frac{1}{2}$ whenever $|\overline \tau|\le \ve_0.$ 
If we define the $\epsilon_0$ neighborhood of $R_{\overline{\mu}}$ in the box metric,
\[R_{\mubar,\ve_0} := [\mu_1-\ve_0,\mu_1 + \ve_0]\times \dotsm\times [\mu_n-\ve_0,\mu_n + \ve_0],\]
then for $\mubar\in\Gamma'\setminus K$ and any $x\in\omega$,
\begin{align*}
\sum\limits_{\{j:\,\overline\lambda_j\in R_{{\mubar},\ve_0}\}} |\Psi\varphi_j(x)|^2 &\le \sum\limits_{\{j:\,\overline\lambda_j\in R_{\mubar,\ve_0}\}}2\beta(\overline \lambda_j - \mubar)|\Psi\varphi_j(x)|^2\\
& \le \sum\limits_{j=0}^\infty 2\beta(\overline \lambda_j - \mubar)|\Psi\varphi_j(x)|^2\\
& = 2\Psi\beta(\overline P - \mubar)\Psi^*(x,x)\\
& \le 2C(1 + \mubar)^{\overline m},
\end{align*}
where the final inequality follows by our hypothesis on $\Psi\beta(\overline P-\mubar)\Psi^*$. Next, note that there exists an $N\in\N$ 
such that for any $\mubar\in\Gamma'\setminus K,$ there exists a finite collection of points $\{\mubar_k\}_{k=1}^N$ such that 
$R_\mubar\subseteq \bigcup\limits_{k=1}^N R_{\mubar_k,\ve_0}$, with the property that $|\mubar-\mubar_k|\le 1$ for all $k$. Then, by 
nonnegativity and our previous calculations, we have that for $\mubar\in\Gamma'\setminus K$ and $x\in\omega$,
\begin{align*}
\Psi\Pi_{R_\mubar}\Psi^*(x,x) & \le \sum\limits_{k=1}^N \Psi\Pi_{R_{\mubar_k,\ve_0}}\Psi^*(x,x)\\
& \le \sum\limits_{k=1}^N 2C(1+\mubar_k)^\mbar\\
& \le C' (1+\mubar)^\mbar
\end{align*}
for some $C' > 0$, which completes the proof. Note that in the above argument it is critical that $C$ and $N$ be independent of both $\mubar,\mubar_k$ and $\ve_0.$
\end{proof}
\bigskip

\subsection{Tauberian Lemma}
The goal of this section is to show that the asymptotics of the rough projector $\Psi\Pi_{\lambda \cbar}\Psi^*(x,y)$ are the same as those of the smoothed projector 
$$\left(\rho\otimes\dotsm  \otimes \rho\right) \ast \Pi_{\lambda, \cbar}\Psi\psi^*(x,y).$$ 
To be more precise, we have the following proposition.  

\begin{prop}\label{taub_prop}
Let $\Lambda$ be the joint spectrum (\ref{eqn:joint_spectrum}) of our QCI system.  Let $(M,g)$ be as in Thm. \ref{thm:main}, 
and $\omega := proj_M(\Omega) \subset M$, with $\Omega$ as in the statement 
of Prop. \ref{prop:smooth_spec_measure_microlocal}. 
Suppose there exists $\Psi\in \Psi^0(M)$, a compact set $K\subset\R^n$, an $\overline m\in\R^n$, and a $C> 0$ 
such that for all $x\in \omega$ and all $\overline \mu \in \overline{p}(W) \setminus K$,
\begin{equation} \label{eqn:on_diag_assump}
\Psi\Pi_{R_{\overline\mu}}\Psi^*(x,x) \le C(1 +\overline \mu)^{\overline m}
\end{equation}
where $R_{\overline{\mu}}$ is as in (\ref{eqn:R_m}). Then, there exists an $\ve_0>0$ such that for any $\rho\in\mathscr S(\R)$ 
with $\supp\wh\rho\subset[-\ve_0,\ve_0]$ and $\wh\rho(s) = 1$ for all $|s|\le \frac{\ve_0}{2}$, there exists a $C_\rho > 0$ 
such that for all $\lambda > 0$ with $\lambda \overline{c} \in (\overline{p}(W) \setminus K) \cap \Lambda^{\complement}$ 
(where we fix $\cbar \in \overline{p}(W) \cap \mathbb S^{n-1}$ having only nonzero components), we have
\[\sup\limits_{x,y\in\omega}\left|\Psi\Pi_{\lambda,\cbar}\Psi^*(x,y) - \left( \rho\otimes\dotsm\otimes\rho \right) 
\ast\Psi\Pi_{\lambda,\cbar}\Psi^*(x,y)\right|  \le C_{\rho'}(1 + \lambda)^{|\overline m| + n - 1}\]
\end{prop}

\begin{rmk} 
    It is always possible to find a conic neighborhood $\Omega$ such that every $\lambda \overline{c} \in \overline{p}(\Omega)$ 
    has only non-zero components: The momentum map is a submersion into $\R^n$ over 
    $\mathscr{O} \subset U \cap \Omega_0$, say as in Prop. \ref{prop:smooth_spec_measure_microlocal}.
    Thus, its image contains
    open subsets not intersecting any of  the coordinate planes of $\mathbb{R}^n$,  and we may take as $\Omega$
    the  inverse image of any such.
\end{rmk}

\begin{cor}\label{cor:7.4}
    Assume the hypotheses of Prop. \ref{taub_prop}.  Set $\Psi = \Op \left( \pi_{\Omega} \right)$.  
    Then for all $\lambda \overline{c} \in \overline{p}(\Omega) \setminus K$ (with $\cbar \in \overline{p}(\Omega) \cap \mathbb S^{n-1}$ 
    having only nonzero components), we have
    \[
        \left( \rho\otimes\dotsm\otimes\rho\right) \ast\Psi\Pi_{\lambda,\cbar}\Psi^*  = \int_{\overline{p}(W) \cap I(\lambda, \overline{c})} 
        \Psi \circ \rho(P_1 - \mu_1 ) \circ \dots \circ  \rho(P_n - \mu_n ) \circ \Psi^* \, d\overline{\mu} \, +\,  \mathcal{O}(1) 
        \]
    for all $x,y \in \omega$.
\end{cor}

\begin{proof}
    This is a simple application of the functional calculus for commuting families of self-adjoint operators followed 
    by Prop. \ref{prop:smooth_spec_measure_microlocal}.  More specifically, we have
    \begin{align*}
        & \rho\otimes\dotsm\otimes\rho\ast\Psi\Pi_{\lambda,\cbar}\Psi^*  = \int_{I(\lambda \overline{c})  \cap \left( \overline{p}(\Omega) \cup 
        \overline{p}(\Omega)^{\complement} \right)} \Psi \circ \rho(P_1 - \mu_1 ) \circ \dots \circ  \rho(P_n - \mu_n ) \circ \Psi^* \, 
        d\overline{\mu} \\
        & = \int_{\overline{p}(W) \cap \left( I(\lambda \overline{c}) \right) } \Psi \circ \rho(P_1 - \mu_1 ) \circ \dots \circ  \rho(P_n - \mu_n ) 
        \circ \Psi^* \, d\overline{\mu} \, \, + \\
        & \, \, \, \int_{\overline{p}(W)^{\complement} \cap \left( I(\lambda \overline{c}) \right)} \Psi \circ \rho(P_1 - \mu_1 ) \circ \dots \circ  
        \rho(P_n - \mu_n ) \circ \Psi^* \, d\overline{\mu} \, \, \\
        &  = \int_{\overline{p}(W) \cap I(\lambda \overline{c})} \Psi \circ \rho(P_1 - \mu_1 ) \circ \dots \circ  \rho(P_n - \mu_n ) \circ \Psi^* \, 
        d\overline{\mu}  + \mathcal{O}(1)
    \end{align*}
    We arrived at the last line by applying Cor. \ref{cor:smooth_measure_neg} to the integrand of the integral in the penultimate line and 
    estimated the resulting quantity.
\end{proof}

We recall that $\Psi\Pi_{\lambda, \overline{c}}\Psi^*(x,y)$ can be expressed as
\[\Pi_{\lambda,\cbar}(x,y) = \sum\limits_{j=0}^\infty\prod\limits_{k=1}^n \mathds 1_{[-\lambda ,\lambda ]}( \lambda_j^{(1)}/|c_1|)
\dotsm\mathds 1_{[-\lambda,\lambda]}(\lambda_j^{(n)}/|c_n|)\Psi\varphi_j(x)\overline{\Psi\varphi_j(y)},\]
and thus by the Fourier inversion formula, the smoothed projector can be written as
\[\left( \rho\otimes\dotsm\otimes\rho \right) \ast \Psi\Pi_{\lambda,\cbar}\Psi^*(x,y) = \sum\limits_{j=0}^\infty\frac{1}{\pi^n}\int\limits_{\R^n}
e^{i\tbar\cdot \lambdabar_j}\wh\rho(t_1)\dotsm\wh\rho(t_n)\frac{\sin(\lambda t_1c_1)}{t_1}\dotsm\frac{\sin(\lambda t_n c_n)}{t_n}
\Psi\varphi_j(x)\overline{\Psi\varphi_j(y)}\,d\tbar.\]
Since the integral in $\tbar$ is separable, we are led to define the function
\begin{equation} \label{eqn:h_lambda}
h_{\lambda}(\mubar):= \prod\limits_{k=1}^n\mathds 1_{[-\lambda,\lambda]}(\mu_k/c_k) 
- \prod\limits_{k=1}^n\frac{1}{\pi}\int\limits_{-\infty}^\infty e^{it_k\mu_k}\wh\rho(t_k)\frac{\sin (\lambda t_kc_k)}{t_k}\,dt_k.
\end{equation}
We claim that a sufficiently good estimates on the values $h_{\lambda}(\lambdabar_j)$ implies the desired bound on 
\[\big|\Psi\Pi_{\lambda,\cbar}\Psi^*(x,y) - \left(\rho\otimes\dotsm\otimes\rho\right)\ast\Psi\Pi_{\lambda,\cbar}\Psi^*(x,y)\big|,\] 
and the proof of this estimate will comprise the bulk of the proof of Prop. \ref{taub_prop}.
In order to bound  $h_{\lambda}$, the following elementary fact about differences of products is useful. 

\begin{lem}\label{lem:products}
Suppose that $a_1,\dotsc,a_n$ and $a_1',\dotsc,a_n'$ are complex numbers. Then, 
\[\left|\prod\limits_{k=1}^n a_k - \prod\limits_{k=1}^n a_k'\right| \le \sum\limits_{\ell=1}^n |a_1|\dotsm 
|a_{\ell-1}|\,|a_\ell-a_\ell'|\, |a_{\ell+1}'|\dotsm |a_n'|.\]
\end{lem}

\begin{proof}
Note that we can write 
\begin{align*}
a_1\dotsm a_n - a_1'\dotsm a_n' &= a_1\dotsm a_n - a_1'a_2\dotsm a_n + a_1'a_2\dotsm a_n -a_1'\dotsm a_n'\\
& = (a_1-a_1')a_2\dotsm a_n + a_1'(a_2\dotsm a_n - a_2'\dotsm a_n').
\end{align*}
Proceeding inductively, we obtain
\[a_1\dotsm a_n - a_1'\dotsm a_n' = \sum\limits_{\ell=1}^n a_1\dotsm a_{\ell-1}(a_\ell-a_\ell')a_{\ell+1}'\dotsm a_n',\]
and an application of the triangle inequality completes the proof. 
\end{proof}

We are now in position to conclude this section with the proof of Prop. \ref{taub_prop},
making use of the convention that, for a group $*$  of one or more indices, 
$C_*$ will denote a constant that is allowed to depend on $*$ but not on other parameters,
and may vary from line to line.

\begin{proof}
Observe that by Cauchy-Schwarz, we have
\begin{align}
\begin{split}\label{h_cauchyschwarz}
&\left|\Psi\Pi_{\lambda,\cbar}\Psi^*(x,y) - \rho\otimes\dotsm\otimes\rho\ast\Psi\Pi_{\lambda,\cbar}\Psi^*(x,y)\right|  \\
& \hskip 1in \le \lp\sum\limits_{j=0}^\infty \left|h_{\lambda}(\overline\lambda_j)\right|\,|\Psi\varphi_j(x)|^2\rp^{\frac{1}{2}}\lp
\sum\limits_{j=0}^\infty \left| h_{\lambda}(\overline\lambda_j)\right| |\Psi\varphi_j(y)|^2\rp^{\frac{1}{2}}.
\end{split}
\end{align}

Before we apply \lemref{lem:products} to $h_{\lambda}$ (\ref{eqn:h_lambda}), we claim that 
\begin{equation}\label{h_simplebound}
\left|\mathds 1_{[-\lambda,\lambda]}(\tau) - \frac{1}{(2\pi)^2}\int\limits_{-\infty}^\infty \wh\rho(t)e^{it\tau}\frac{\sin(\lambda t)}{t}\,dt\right| 
\le C_N (1 + \big||\tau|-\lambda\big|)^{-N}
\end{equation}
for any $N\in\N.$ To see this, we recall that $\frac{\sin(\lambda t)}{t}$ is the Fourier transform of $\mathds 1_{[-\lambda,\lambda]}(\cdot)$, 
and hence by properties of the Fourier transform, 
\[\frac{1}{2\pi}\int\limits_{-\infty}^\infty e^{it\tau}\wh\rho(t)\frac{\sin(\lambda t)}{t}\,dt = \int\limits_{-\infty}^\infty \rho(\tau - \tau')\mathds 1_{[-
\lambda,\lambda]}(\tau')\,d\tau' = \int\limits_{-\lambda}^\lambda \rho(\tau-\tau')\,d\tau' = \int\limits_{\tau-\lambda}^{\tau+\lambda}\rho(\tau')
\,d\tau'.\] 
We now recall our discussion from Sec. \ref{sect:joint_proj_wave}.  When $|\tau| > \lambda,$ we use that 
$\mathds 1_{[-\lambda,\lambda]}(\tau) = 0$ and the fact that $\rho$ is Schwartz to obtain the desired bound. 
Similarly, when $|\tau| < \lambda$, we use that $\mathds 1_{[-\lambda,\lambda]}(\tau) = 1$, combined with the 
Schwartz properties of $\rho$ and the fact that $\int_{-\infty}^\infty\rho(\tau')\,d\tau' = 1$. 
When $|\tau| = \lambda$, \eqref{h_simplebound} merely claims that the difference is uniformly bounded in $\lambda$, 
which is clear by inspection. 

Now, for $1\le k\le n$, set
\[a_k = \mathds 1_{[-\lambda,\lambda]}(\mu_k/c_k) \quad \text{and}\quad a_k' = \frac{1}{\pi}\int\limits_{-\infty}^\infty 
\wh\rho(t_k)e^{it_k\mu_k}\frac{\sin(\lambda t_k c_k)}{t_k}\,dt_k.\]
Then, by \eqref{h_simplebound}, if we choose any $N\in \N$, there will exist constants $C_{N,k} > 0$ such that
$|a_k-a_k'|\le C_{N,k}(1+\big||\mu_k|-\lambda|c_k|\big|)^{-N}$; letting $C_N=\max_{1\le k\le n}C_{N,k}$, we thus have
$$|a_k-a_k'|\le C_{N}(1+\big||\mu_k|-\lambda|c_k|\big|)^{-N},\quad 1\le k\le n.$$ 
Applying \lemref{lem:products} to $h_\lambda(\mubar)$ from (\ref{eqn:h_lambda}), one sees  that 
\begin{equation}\label{e:h_a'}
\left|h_\lambda(\mubar)\right| \le \sum\limits_{l = 1}^n |a_1|\,\dotsm\,|a_{\ell-1}|\, \left( C_{N}(1+\big||\mu_l|-\lambda|c_l|\big|)^{-N}\right) \, 
|a_{\ell+1}'|\,\dotsm\,|a_n'|.
\end{equation}
Note that we can also write 
\[a_k' = a_k + \mathcal O\lp(1 + \big||\mu_k|-\lambda|c_k|\big|)^{-N}\rp.\]
Applying this to each factor of $a_k'$ in \eqref{e:h_a'}, we obtain an upper bound on $h_\lambda(\mubar)$ which consists of a finite 
linear combination of terms, each of which is a product of the form 
\begin{equation}\label{e:h_terms}
\prod\limits_{\ell\in\mathscr N_1}\mathds 1_{[-\lambda,\lambda]}(\mu_\ell/c_\ell)
\prod\limits_{k\in\mathscr N_2}(1+\big||\mu_k|-\lambda|c_k|\big|)^{-N},
\end{equation}
where $\mathscr N_1,\mathscr N_2$ form a partition of the indices $\{1,\dotsc,n\}.$ 

We want to use \eqref{e:h_terms} to estimate $h_\lambda(\lambdabar_j)$, so set $\mubar = \lambdabar_j$,
i.e., $\mu_k=\lambda_j^{(k)},\, 1\le k\le n$, and  then fix a $k\in \{1,\dotsc,n\}$.
We consider first the case where $\mathscr N_2$ is a singleton, which we claim is the worst  case.

Observe  that there exist constants  $C_{N,k}> 0$ (which can vary from line to line)  such that 
\begin{align*}
&\sum\limits_{j=0}^\infty \prod\limits_{\substack{\ell=1\\ 
\ell\ne k}}^n C_{N,k}\mathds 1_{[-\lambda,\lambda]}(\lambda_j^{(\ell)}/c_\ell)
(1+\big||\lambda_j^{(k)}|-\lambda|c_k|\big|)^{-N}|\Psi\varphi_j(x)|^2 \\
& \hskip 0.5in \le C_{N,k} \sum\limits_{\overline{s} \in\Z^n}\sum\limits_{\lambdabar_j\in R_{\overline{s}}} \prod\limits_{\substack{\ell=1\\ 
\ell\ne k}}^n \mathds 1_{[-\lambda,\lambda]}(\lambda_j^{(\ell)}/c_\ell)(1+\big||\lambda_j^{(k)}|-\lambda|c_k|\big|)^{-N}|\Psi\varphi_j(x)|^2\\
& \hskip 0.5in \le C_{N,k} \sum\limits_{s_1=-\floor{\lambda|c_1|}}^{\floor{\lambda |c_1|}}\underbrace{\dotsm}_{\text{skip} s_k}
\sum\limits_{s_n=-\floor{\lambda|c_n|}}^{\floor{\lambda|c_n|}} \sum\limits_{s_k\in\Z}  \sum\limits_{\lambdabar_j\in R_{\overline{s}}} (1+
\big||\lambda_j^{(k)}|-\lambda|c_k|\big|)^{-N}|\Psi\varphi_j(x)|^2 \\
& \hskip 0.5in \le C_{N,k} \sum\limits_{s_1=-\floor{\lambda|c_1|}}^{\floor{\lambda |c_1|}}\underbrace{\dotsm}_{\text{skip} s_k}
\sum\limits_{s_n=-\floor{\lambda|c_n|}}^{\floor{\lambda|c_n|}} \sum\limits_{s_k\in\Z} (1 + \big||s_k| - \lambda|c_k|\big|)^{-N}
\sum\limits_{\lambdabar_j\in R_{\overline{s}}}|\Psi\varphi_j(x)|^2
\\
& \hskip 0.5in = C_{N,k} \sum\limits_{s_1=-\floor{\lambda|c_1|}}^{\floor{\lambda |c_1|}}\underbrace{\dotsm}_{\text{skip} s_k}
\sum\limits_{s_n=-\floor{\lambda|c_n|}}^{\floor{\lambda|c_n|}} \sum\limits_{s_k\in\Z} (1 + \big||s_k| - \lambda|c_k|\big|)^{-N} 
\Psi\Pi_{R_{\overline{s}}}\Psi^*(x,x)
\end{align*}

Take $N\ge 2$ so that the above summation over $s_k\in\Z$ is finite (again, for a fixed $k$). 
Recalling the definition of $R_{\overline{\mu}}$ in (\ref{eqn:R_m}), 
we can combine the bound 
(\ref{eqn:on_diag_assump}) on 
$\Psi\Pi_{R_\mubar}\Psi(x,x)$  with  the above calculations to obtain 
\[\sum\limits_{j=0}^\infty \prod\limits_{\substack{\ell=1\\ \ell\ne k}}^n C_{N,k}\mathds 1_{[-\lambda,\lambda]}(\lambda_j^{(\ell)}/c_\ell)
(1+\big||\lambda_j^{(k)}|-\lambda|c_k|\big|)^{-N}|\Psi\varphi_j(x)|^2 \le \wt C_k \lambda^{|\overline{m}| + n-1}\]
for some $\wt C_k > 0$.  Now, notice that if we repeat this argument for the remaining terms in the upper bound for $h_\lambda$, each 
additional factor of $(1+\big||\lambda_j^{(k)}| - \lambda|c_k|\big|)^{-N}$ would yield one lower power of $\lambda$ in the final estimate. 
Therefore, since $\max_{1\le k \le n} \wt C_k <\infty$, we have that 
\[\sum\limits_{j=0}^\infty |h_\lambda(\lambdabar_j)||\Psi\varphi_j(x)|^2 \le C \lambda^{|\overline{m}|+n-1}\]
for some $C.$ 

Now, if $\mathcal N_2=\{k_1,\dots,k_p\}$ has more than one element, 
then the $\prod\limits_{\substack{\ell=1\\ \ell\ne k}}^n$ above becomes
$\prod\limits_{\substack{\ell=1\\ \ell\ne k_1,\dots,k_p}}^n$.
We can then let $k$ be any one of $k_1,\dots,k_p$ and apply the argument above, with the  factors coming from the other $k_i$
being bounded and thus harmless, resulting in a contribution dominated by

Thus, by \eqref{h_cauchyschwarz}, we have that 
\[\big|\Psi\Pi_{\lambda\cbar}\Psi^*(x,y) - \left(\rho\otimes\dotsm\otimes\rho\right)\ast\Psi\Pi_{\lambda\cbar}\Psi^*(x,y)\big| 
\le C \lambda^{|\overline{m}|+n-1}.\]
which completes the proof of Prop. \ref{taub_prop}.
\end{proof}


\subsection{Final steps of the proof of Theorem \ref{thm:main}} \label{sect:final_steps}

To conclude the argument that leads to Thm. \ref{thm:main}, we want to show that the hypotheses of 
Prop. \ref{taub_prop} are satisfied.  

We start by recalling  (\ref{eqn:Omega_set}), and and set $\Psi = Q_{\Omega}$
in Prop. \ref{prop:smooth_spec_measure_microlocal}.  Choose $\overline{c} \in \mathbb S^{n-1}$ as in the hypotheses of 
Prop. \ref{taub_prop}, 
and note that the asymptotics for $\left( \rho\otimes\dotsm\otimes\rho \right) \ast\Psi\Pi_{\lambda\cbar}\Psi^*(x,y)$ 
are obtained by integrating (\ref{eqn:smooth_measure_asymps}) over $I(\lambda, \overline{c}) \cap \overline{p}(\Omega)$ in the 
$\overline{\mu}$ variable, considering that $\Psi$ is supported microlcally on $\Omega$.
Set $\delta_0 = 3\epsilon_0/4$ as in (\ref{eqn:Omega_set});
by the result of  Prop. \ref{prop:smooth_spec_measure_microlocal} for $\overline{m} = \overline{0}$,
the hypotheses of  Lemma \ref{l:smoothed_measure_rough_measure} are thus satisfied. 
Taking all of this together, we see that the hypotheses of Prop. \ref{taub_prop} are satisfied.  
Finally, note that the remainder term $\mathscr{R}$ in the conclusion of Theorem \ref{thm:main} 
is, by the triangle inequality, dominated by the sum of the integral in $\xi$ of 
$\mathcal{R}$ from (\ref{eqn:smooth_measure_asymps}) over the box $I(\lambda;\overline{c})$ 
(which can be estimated using Cor. \ref{cor:7.4})
and the term $C_{\rho'}(1 + \lambda)^{n-1}$ from Prop. \ref{taub_prop} with $\overline{m}=0$.  
\qed


\subsection{Proof of Corollary \ref{cor:main}} \label{sect:cor proof}

Given Thm. \ref{thm:main}, we expand the phase function in \eqref{eq:QCI_WeylLaw}
about the diagonal $\{x=y\}$, writing
$$S(x,\xi)-S(y,\xi)=A(x,y,\xi)\cdot (x-y) =\left(\nabla_xS(x,\xi)+T(x,y,\xi)\right)\cdot(x-y)$$
where $T=\mathcal O\left(|x-y|\cdot|\xi|\right)$. This gives 
$$e^{i\left(S(x,\xi)-S(y,\xi)\right)}=e^{i\nabla_xS(x,\xi)\cdot(x-y)}(1+\mathcal O\left(|x-y|^2|\xi|\right).$$
Substituting the first term into \eqref{eq:QCI_WeylLaw} and using the fact that $b\cdot a=|\sigma(\Psi)|^2$ on the diagonal
yields the first term in \eqref{eqn:cor}. On the other hand, substituting the second term into \eqref{eq:QCI_WeylLaw}
yields an integral which can simply be estimated by
$$vol\left(I(\lambda,\overline{c})\right)\cdot \mathcal O(1)\cdot \mathcal O\left(|x-y|^2|\lambda|\right) 
=\mathcal O\left(|x-y|^2\lambda^{n+1}\right),$$
which is $\mathcal O\left(\lambda^{n-1}\right)$ since Corollary \ref{cor:main} assumes that $|x-y|\le c/\lambda$,
finishing the proof. \qed

\bibliography{master_bib}

\begin{thebibliography}{MathOver17}
\expandafter\ifx\csname url\endcsname\relax
  \def\url#1{\texttt{#1}}\fi
\expandafter\ifx\csname doi\endcsname\relax
  \def\doi#1{\burlalt{doi:#1}{http://dx.doi.org/#1}}\fi
\expandafter\ifx\csname urlprefix\endcsname\relax\def\urlprefix{URL }\fi
\expandafter\ifx\csname href\endcsname\relax
  \def\href#1#2{#2}\fi
\expandafter\ifx\csname burlalt\endcsname\relax
  \def\burlalt#1#2{\href{#2}{#1}}\fi

\bibitem[Ava56]{Avakumovic1956}
V.~G. Avakumovi\'{c}.
\newblock \"{U}ber die {E}igenfunktionen auf geschlossenen {R}iemannschen
  {M}annigfaltigkeiten.
\newblock {\em Math. Z.}, 65:327--344, 1956.
\newblock \doi{10.1007/BF01473886}.

\bibitem[Ber77]{Berard1977}
P.~H. B{\'e}rard.
\newblock On the wave equation on a compact {R}iemannian manifold without
  conjugate points.
\newblock {\em Mathematische Zeit.}, 155(3):249--276, 1977.
\newblock \doi{10.1007/BF02028444}.

\bibitem[BGT07]{BGT07}
N.~Burq, P.~G\'{e}rard, and N.~Tzvetkov.
\newblock Restrictions of the {Laplace-Beltrami} eigenfunctions to
  submanifolds.
\newblock {\em Duke Math. Jour.}, 136(1):445--486, 2007.

\bibitem[BS07]{BS02}
E.~Bogomolny and C.~Schmit.
\newblock Random wavefunctions and percolation.
\newblock {\em J. Phys. A}, 40(47):14033--14043, 2007.
\newblock \doi{10.1088/1751-8113/40/47/001}.

\bibitem[CdV79]{CdV79}
Y.~Colin~de Verdi\`ere.
\newblock Spectre conjoint des op\'erateurs pseudo-differential qui commutent:
  Le case non-integrable.
\newblock {\em Duke Math. Jour.}, 46(1):169--182, 1979.

\bibitem[CH15]{CanzaniHanin2015}
Y.~Canzani and B.~Hanin.
\newblock {Scaling limit for the kernel of the spectral projector and remainder
  estimates in the pointwise Weyl law}.
\newblock {\em Anal. PDE}, 8(7):1707--1731, 2015.
\newblock \doi{10.2140/apde.2015.8.1707}.

\bibitem[CH18]{CanzaniHanin2018}
Y.~Canzani and B.~Hanin.
\newblock ${{C^{\infty}}}$ {Scaling asymptotics for the spectral projector of
  the Laplacian}.
\newblock {\em Jour. Geometric Analysis}, 28(1):111--122, 2018.
\newblock \doi{10.1007/s12220-017-9812-5}.

\bibitem[DG75]{DuistermaatGuillemin1975}
J.~Duistermaat and V.~Guillemin.
\newblock The spectrum of positive elliptic operators and periodic
  bicharacteristics.
\newblock {\em Inventiones math.}, 29:39--79, 1975.
\newblock \doi{10.1007/BF01405172}.

\bibitem[DH71]{DuistermaatHormander72}
J.~Duistermaat and L.~H\"ormander.
\newblock Fourier integral operators, ii.
\newblock {\em Acta mathematica}, 128:183--269, 1971.

\bibitem[Dui74]{Duistermaat74}
J.~J. Duistermaat.
\newblock Oscillatory integrals, {L}agrange immersions and unfolding of
  singularities.
\newblock {\em Comm. Pure Appl. Math.}, 27:207--281, 1974.
\newblock \doi{10.1002/cpa.3160270205}.

\bibitem[GT20]{GalkowskiToth2020}
J.~Galkowski and J.~A. Toth.
\newblock Pointwise bounds for joint eigenfunctions of quantum completely
  integrable systems.
\newblock {\em Comm. Math. Phys.}, 375(2):915--947, 2020.
\newblock \doi{10.1007/s00220-020-03730-3}.

\bibitem[H{\"o}r68]{Hormander1968}
L.~H{\"o}rmander.
\newblock The spectral function of an elliptic operator.
\newblock {\em Acta math.}, 121(1):193--218, 1968.
\newblock \doi{10.1007/BF02391913}.

\bibitem[H\"or71]{FIO1}
L.~H\"{o}rmander.
\newblock Fourier integral operators. {I}.
\newblock {\em Acta Math.}, 127(1-2):79--183, 1971.
\newblock \doi{10.1007/BF02392052}.

\bibitem[H{\"o}rIII]{HormanderBook1985}
L.~H{\"o}rmander.
\newblock {\em The Analysis of Linear Partial Differential Operators III}.
\newblock Berlin: Springer-Verlag, 1985.
\newblock \doi{10.1007/978-3-540-49938-1}.

\bibitem[H{\"o}rIV]{HormanderBook1985b}
L.~H{\"o}rmander.
\newblock {\em The Analysis of Linear Partial Differential Operators IV}.
\newblock Berlin: Springer-Verlag, 1985.
\newblock \doi{10.1007/978-3-540-49938-1}.

\bibitem[Hu09]{Hu2009}
R.~Hu.
\newblock {$L^p$} norm estimates of eigenfunctions restricted to submanifolds.
\newblock {\em Forum Math.}, 21(6):1021--1052, 2009.
\newblock \doi{doi.org/10.1515/FORUM.2009.051}.

\bibitem[{Kee}23]{Keeler2023}
B.~{Keeler}.
\newblock A logarithmic improvement in the two-point weyl law for manifolds
  without conjugate points.
\newblock {\em Annales de l'institut Fourier}, march 2023.
\newblock To appear. Available from:
  \href{https://arxiv.org/abs/1905.05136}{https://arxiv.org/abs/1905.05136}.

\bibitem[Lev52]{Levitan1952}
B.~M. Levitan.
\newblock On the asymptotic behavior of the spectral function of a self-adjoint
  differential equation of the second order.
\newblock {\em Izv. Akad. Nauk SSSR. Ser. Mat.}, 16(4):325--352, 1952.

\bibitem[Lev55]{Levitan1955}
B.~M. Levitan.
\newblock On the asymptotic behavior of a spectral function and on expansion in
  eigenfunctions of a self-adjoint differential equation of second order. {II}.
\newblock {\em Izv. Akad. Nauk SSSR. Ser. Mat.}, 19:33--58, 1955.

\bibitem[Mar16]{Marsh}
S.~Marshall.
\newblock {$L^p$} norms of higher rank eigenfunctions and bounds for spherical
  functions.
\newblock {\em J. Eur. Math. Soc. (JEMS)}, 18(7):1437--1493, 2016.
\newblock \doi{10.4171/JEMS/619}.

\bibitem[MathOver17]{Mathoverflow_Momentum}
{\em Mathoverflow: Origin of the name ``momentum map"}.
\newblock
  \urlprefix\url{https://mathoverflow.net/questions/242468/origin-of-the-name-momentum-map}.

\bibitem[NS09]{NS09}
F.~Nazarov and M.~Sodin.
\newblock On the number of nodal domains of random spherical harmonics.
\newblock {\em Amer. J. Math.}, 131(5):1337--1357, 2009.
\newblock \doi{10.1353/ajm.0.0070}.

\bibitem[NS16]{NS16}
F.~Nazarov and M.~Sodin.
\newblock Asymptotic laws for the spatial distribution and the number of
  connected components of zero sets of {G}aussian random functions.
\newblock {\em J. Math. Phys. Anal. Geom.}, 12(3):205--278, 2016.
\newblock \doi{10.15407/mag12.03.205}.

\bibitem[Roz17]{Roz17}
Y.~Rozenshein.
\newblock The number of nodal components of arithmetic random waves.
\newblock {\em Int. Math. Res. Not. IMRN}, (22):6990--7027, 2017.
\newblock \doi{10.1093/imrn/rnw226}.

\bibitem[Saf88]{Safarov1988}
Y.~G. Safarov.
\newblock Asymptotics of a spectral function of a positive elliptic operator
  without a nontrapping condition.
\newblock {\em Funktsional. Anal. i Prilozhen.}, 22(3):53--65, 96, 1988.
\newblock \doi{10.1007/BF01077627}.

\bibitem[SarMor]{Sarnak_letter}
P.~Sarnak.
\newblock {\em Letter to Morawetz}.
\newblock
  \urlprefix\url{https://publications.ias.edu/sites/default/files/Sarnak_Letter_to_Morawetz.pdf}.

\bibitem[Sog14]{SoggeBook2014}
C.~Sogge.
\newblock {\em Hangzhou Lectures on Eigenfunctions of the Laplacian (AM-188)}.
\newblock Princeton University Press, 2014.
\newblock \doi{10.1515/9781400850549}.

\bibitem[Sog17]{SoggeBook2017}
C.~Sogge.
\newblock {\em Fourier integrals in classical analysis}, volume 210.
\newblock Cambridge University Press, 2017.
\newblock \doi{10.1017/9781316341186}.

\bibitem[Str72]{Strichartz1972}
R.~Strichartz.
\newblock A functional calculus for elliptic pseudo-differential operators.
\newblock {\em American Journal of Mathematics}, 94(3):711--722, 1972.

\bibitem[SZ02]{SoggeZelditch2002}
C.~Sogge and S.~Zeldtich.
\newblock Riemannian manifolds with maximal eigenfunction growth.
\newblock {\em Duke Math. J.}, 114(3):387--437, 2002.
\newblock \doi{10.1215/S0012-7094-02-11431-8}.

\bibitem[Tac19]{Tacy2019}
M.~Tacy.
\newblock {$L^p$} estimates for joint quasimodes of semiclassical
  pseudodifferential operators.
\newblock {\em Israel J. Math.}, 232(1):401--425, 2019.
\newblock \doi{10.1007/s11856-019-1878-2}.

\bibitem[Tot95]{TothVarious}
J.~Toth.
\newblock Various quantum mechanical aspects of quadratic forms.
\newblock {\em Jour. Functional Analysis}, 130:1--42, 1995.

\bibitem[Toth09]{Toth09}
J.~A. Toth.
\newblock {$L^2$}-restriction bounds for eigenfunctions along curves in the
  quantum completely integrable case.
\newblock {\em Comm. Math. Phys.}, 288(1):379--401, 2009.
\newblock \doi{10.1007/s00220-009-0747-y}.

\bibitem[TZ03]{TZ01}
J.~A. Toth and S.~Zelditch.
\newblock Norms of modes and quasi-modes revisited.
\newblock In {\em Harmonic analysis at {M}ount {H}olyoke ({S}outh {H}adley,
  {MA}, 2001)}, volume 320 of {\em Contemp. Math.}, pages 435--458. Amer. Math.
  Soc., Providence, RI, 2003.
\newblock \doi{10.1090/conm/320/05622}.

\bibitem[Wig23]{Wig03}
I.~Wigman.
\newblock Random wavefunctions and percolation.
\newblock {\em J Appl. and Comput. Topology}, (47), 2023.
\newblock \doi{10.1007/s41468-023-00140-x}.

\end{thebibliography}
\bibliographystyle{halpha-abbrv}

\end{document}